\documentclass[10pt]{amsart}

\usepackage{times,amsmath,amsbsy,amssymb,amscd,mathrsfs}

\newcommand{\breakline}{
\begin{center}
------------------------------------------------------------------------------------------------------------
\end{center}
}

\newtheorem{theorem}{Theorem}[section]
\newtheorem{lemma}[theorem]{Lemma}

\newtheorem{remark}[theorem]{Remark}

\newcommand{\dd}{\,{\rm d}}
\newcommand{\osc}{\,{\rm osc}}

\newcommand{\mbb}{\mathbb}

\newcommand{\bs}{\boldsymbol}
\newcommand{\mcal}{\mathcal}

\DeclareMathOperator*{\img}{img}
\newcommand{\curl}{{\rm curl\,}}
\renewcommand{\div}{\operatorname{div}}
\newcommand{\grad}{{\rm grad\,}}
\DeclareMathOperator*{\tr}{tr}

\usepackage{diagrams}

\begin{document}
\title[Convergence of AMFEM for the Hodge Laplacian Equation]{Convergence of Adaptive Mixed Finite Element Methods for the Hodge Laplacian Equation: without harmonic forms}

\author{Long Chen and Yongke Wu}

\address[L.~Chen]{Department of Mathematics, University of California at Irvine, Irvine, CA 92697, USA}
\email{chenlong@math.uci.edu}

\address[Y.~Wu]{School of Mathematical Sciences, University of Electronic Science and Technology of China, Chengdu 611731, China.}
\email{wuyongkeuestc@126.com}


\date{\today}
\begin{abstract}
Finite element exterior calculus (FEEC) has been developed as a systematical framework for constructing and analyzing stable and accurate numerical method for partial differential equations by employing differential complexes. This paper is devoted to analyze the convergence of adaptive mixed finite element methods for the Hodge Laplacian equations based on FEEC without considering harmonic forms. A residual type \emph{a posteriori} error estimates is obtained by using the Hodge decomposition, the regular decomposition and bounded commuting quasi-interpolants. An additional marking strategy is added to ensure the quasi-orthogonality, based on which the convergence of adaptive mixed finite element methods is obtained without assuming the initial mesh size is small enough.
\end{abstract}

\maketitle

\section{Introduction}

This paper is devoted to analyze the convergence of adaptive mixed finite element methods (AMFEM) for the Hodge Laplacian equations based on finite element exterior calculus (FEEC) which is a general framework for constructing and analyzing mixed finite element methods for partial differential equations developed by Arnold, Falk, and Winther~\cite{Arnold.D;RS.Falk;Winther.R2006,Arnold.D;RS.Falk;Winther.R2010}. 

Let 
\begin{equation}\label{short}
H_0\Lambda^-(\Omega) \stackrel{\dd ^-}{\longrightarrow} H_0\Lambda(\Omega) \stackrel{\dd}{\longrightarrow} H_0\Lambda^+(\Omega)    
\end{equation}
be part of the de Rham sequence where $H_0^-\Lambda(\Omega),\ H_0\Lambda(\Omega),\text{ and }H_0^+\Lambda(\Omega)$ are appropriate spaces of differential forms and $\dd, \dd^{-}$ are exterior derivatives. Assume that $\Omega$ is a domain with trivial cohomology and thus the sequence \eqref{short} is exact, i.e., $\ker(\dd) = \img (\dd ^-)$. For readers who are more familiar with the vector calculus, a concrete example is 
\begin{equation}\label{shortvector}
H_0^1(\Omega) \stackrel{\grad}{\longrightarrow} \bs H_0(\curl;\Omega) \stackrel{\curl}{\longrightarrow} \bs H_0(\div; \Omega).    
\end{equation}
Consider the mixed formulation of the Hodge Laplacian corresponding to this exact sequence: Given $f\in L^2\Lambda(\Omega)$, find $(\sigma,u) \in H_0\Lambda^-(\Omega) \times H_0\Lambda(\Omega)$, such that 
\begin{equation}\label{continuous}
\left\{
\begin{array}{llrlll}
\langle\sigma,\tau\rangle & - & \langle \dd ^- \tau, u\rangle & = & 0 & \forall \ \tau \in H_0\Lambda^-(\Omega),\\
\langle \dd ^- \sigma,v\rangle & + & \langle \dd u , \dd v\rangle & = & \langle f, v \rangle & \forall  \ v \in H_0\Lambda(\Omega).
\end{array}  
\right. 
\end{equation}
Given a shape regular triangulation $\mathcal T_h$ of $\Omega$, let $V_h^{-} \subset H_0\Lambda^-(\Omega)$, $V_h \subset H_0\Lambda(\Omega), \text{ and } V_h^+ \subset H_0\Lambda^+(\Omega)$ be the finite element spaces so that the discrete sequence $V_h^{-} \stackrel{\dd ^-}{\longrightarrow}  V_h\stackrel{\dd}{\longrightarrow}  V_h^+$ is also exact.
We are interested in the accuracy of the following mixed finite element approximation to \eqref{continuous}: Find $(\sigma_h,u_h) \in V_h^- \times V_h$, such that
\begin{equation}\label{discrete}
\left\{
\begin{array}{llrlll}
\langle\sigma_h,\tau_h\rangle & - & \langle \dd ^- \tau_h, u_h\rangle & = & 0 & \forall \tau_h \in V_h^{-},\\
\langle \dd ^- \sigma_h,v_h\rangle & + & \langle \dd u_h , \dd v_h\rangle & = & \langle f, v_h \rangle & \forall  v_h \in V_h.
\end{array}
\right.
\end{equation}  
Precise definition of spaces and differential operators can be found in Section 2. For the specific sequence \eqref{shortvector}, \eqref{continuous} is the mixed formulation of the vector Laplacian using $H_0^1(\Omega)$ for $\sigma$ and $\bs H_0(\curl;\Omega)$ for $u$ and \eqref{discrete} is the discretization using the Lagrange element for $\sigma_h$ and the edge element for $u_h$.

When the domain $\Omega$ is non-smooth or non-convex, the solution of the Hodge Laplacian equation \eqref{continuous} usually contains singularity, and consequently the finite element approximation \eqref{discrete} based on quasi-uniform grids will not provide optimal order of convergence. Adaptive finite element methods (AFEM) based on the iteration
\begin{equation}\label{eq:algo_pro}
\text{SOLVE} \rightarrow \text{ESTIMATE} \rightarrow \text{MARK} \rightarrow \text{REFINE}
\end{equation}
are an effective procedure to achieve quasi-optimal order of convergence. In this procedure, SOLVE computes the solution $(\sigma_l,u_l)$ of the discrete problem \eqref{discrete} discretized on the grid $\mathcal T_l$. The step ESTIMATE computes an \emph{a posteriori} error estimators $\eta((\sigma_l,u_l),K)$ for each element $K \in \mathcal T_l$. These estimators are then used in the step MARK to select a subset $\mathcal M_l$ of $\mathcal T_l$. Finally, the step REFINE refines each element in $\mathcal M_l$ and makes a completion to get a shape regular and conforming mesh $\mathcal T_{l+1}$.
 
Guaranteeing the convergence of the solution sequence produced by the adaptive procedure \eqref{eq:algo_pro} is a fundamental question and has been an active research topic in recent decades. Babu\v{s}ka and Vogelius \cite{Babuska;Vogelius1984} started the convergence analysis of AFEM in one dimension. D\"orfler~\cite{Dorfler1996} introduced a crucial marking strategy which will be called D\"orfler marking afterwards, i.e., for a given $\theta\in (0,1)$, mark $\mcal M_l$ s.t. 
$$
\sum_{K\in \mcal M_l}\eta^2((\sigma_l,u_l),K)\geq \theta \sum_{K\in \mcal T_l}\eta^2((\sigma_l,u_l),K).
$$
With this marking strategy, he proved strict energy norm reduction for the Poisson equation under the assumption that the initial mesh was sufficiently fine. Morin, Nochetto, and Siebert~\cite{Morin;Nochetto;Siebert;2000,Morin;Nochetto;Siebert;2002} involved the so-called interior node property and an additional marking step of oscillation to remove the initial mesh condition. The requirement of the interior node property and the extra marking of the oscillation was further removed by Cascon, Kreuzer, Nochetto, and Siebert~\cite{Cascon;Kreuzer2008}.
We refer to~\cite{Nochetto;Siebert;Veeser2009} for a survey of convergence analysis of AFEM for elliptic problems
and~\cite{Feischl;Fuhrer;Pratetorius2014,Carstensen;Feischl;Praetoriu2014} for recent advance of this topic.

We shall study the convergence of AMFEM for the Hodge Laplacian based on FEEC.
By treating the Hodge Laplacian for general differential forms, we unify convergence proofs of adaptive finite element methods of many important second order elliptic equations involving the well-known primary and mixed formulation of the scalar elliptic equation~\cite{Dorfler1996,Morin;Nochetto;Siebert;2000,Morin;Nochetto;Siebert;2002,Stevenson2007,Cascon;Kreuzer2008,Chen;Holst;Xu2009,Huang;Xu2012,Feischl;Fuhrer;Pratetorius2014,Carstensen;Feischl;Praetoriu2014}. For the vector Laplacian and related problems, convergence and optimality of AFEM for Maxwell's equations can be found in the literature; see, for example,~\cite{Chen;Xu;Zou2009,Zhong;Chen;Shu;Xu2012,Chen;Xu;Zou2012,Duan;Qiu;Tan;Zheng2016}. 

For the Hodge Laplacian problem, an important step is the unified residual-type {\it a posteriori} error estimators developed by Demlow and Hirani~\cite{Demlow2014}.
Their \emph{a posteriori} error estimators were obtained by using the inf-sup condition and the continuity of \eqref{continuous} in the $H\Lambda$-norm and in the form
\begin{equation}\label{DH}
\|u-u_h\|_{H\Lambda} + \|\sigma - \sigma_h\|_{H\Lambda^-} \leq C \left(\eta_{-1}(\sigma_h, u_h) + \eta_{0}(\sigma_h, u_h) \right),
\end{equation}
where $\eta_{-1}$ is to control the $L^2$-norm $\|\sigma- \sigma_h\| + \| u -u_h\|$, and $\eta_{0}$ for the  energy norm $\|\dd (u - u_h)\| + \|\dd ^{-}(\sigma - \sigma_h)\|$. See \cite[Lemma 7 and 8]{Demlow2014} for the precise definition of these two estimators.

A key ingredient in the convergence theory of AFEM is a certain orthogonality of the error to the finite element space which is not revealed in~\cite{Demlow2014}. The first main contribution of this paper is to explore some full and partial orthogonality contained in the mixed formulation of the Hodge Laplacian. The orthogonality of $\sigma$
can be derived by choosing test functions $v_h = \dd ^- \tau_h$ in \eqref{discrete} and using the fact $\dd\circ \dd ^{-} = 0$ to get
$$ 
\langle \dd ^-(\sigma - \sigma_h), \dd ^-\tau_h\rangle = 0, \quad \forall \tau_h \in V_h^-.
$$ 
The orthogonality of $u$ is less straightforward. Let $\mathcal T_H$ and $\mcal T_h$ be two triangulations with $\mcal T_h$ being a refinement of $\mcal T_H$. Let $(\sigma_h, u_h)$ and $(\sigma_H, u_H)$ be the mixed finite element approximations in the spaces based on $\mcal T_h$ and $\mcal T_H$ respectively. We shall prove that 
\begin{equation}\label{ucross}
\langle \dd (u - u_h), \dd (u_h - u_H)\rangle \leq 2Ch^s\|\dd ^-(\sigma - \sigma_h)\| \|\dd (u_h - u_H)\|,
\end{equation}
where $1/2\leq s\leq 1$ is a regularity constant depending only on the shape of the domain $\Omega$ and $h$ is the maximum of the diameter of simplices in $\mathcal T_h$. To prove \eqref{ucross}, we generalize a technique used for Maxwell's equations, i.e., a discrete divergent free field can be lifted to a divergence free field within error $h^s$, to the setting of general differential forms, see Lemma \ref{lem:app_Q}. 
Inequality \eqref{ucross} will imply the following quasi-orthogonality 
\begin{equation}\label{hsquasiorth}
\begin{array}{ll}
&(1-Ch^s)\|\dd ^-(\sigma - \sigma_h)\|^2 + \|\dd (u - u_h)\|^2  \\
&\leq \|\dd ^-(\sigma - \sigma_H)\|^2 + \|\dd (u - u_H)\|^2 \\
& - \left [\|\dd ^-(\sigma - \sigma_H)\|^2 + (1+Ch^s)\|\dd (u - u_H)\|^2\right ].
\end{array}
\end{equation}
With such a quasi-orthogonality, the convergence and optimality of AMFEM can be obtained under the assumption that the initial mesh size is sufficiently small as has been done in the literature for Maxwell's equations~\cite{Zhong;Chen;Shu;Xu2012, Chen;Xu;Zou2012,Duan;Qiu;Tan;Zheng2016}. 

The second main contribution of this paper is to remove this assumption for the convergence proof. To better present the result, we now change the subscript to the iteration index. By enforcing an additional D\"orfler marking for the error estimator $\eta(\sigma_l)$ of $\sigma$ only, we can prove the reduction of this error estimator and thus inequality \eqref{ucross} can be improved to 
$$\langle \dd (u - u_{l+m}), \dd (u_{l+m} - u_{m})\rangle \leq \epsilon \eta(\sigma_{l})\|\dd (u_{l+m} - u_{l})\|$$
for an arbitrary small constant $\epsilon$ provided $m$ is large enough. Then this reduction is used to prove the convergence of the total error plus the error estimator 
$$
\|\dd ^-(\sigma - \sigma_k)\|^2 + \|\dd (u - u_k)\|^2 + \alpha \, \eta^2(\sigma_k, u_k),
$$
with a suitable weight $\alpha$.  We prove the convergence by showing a convergent sub-sequence with a bounded gap in the iteration index.

One more contribution is our new way to derive the residual type \emph{a posteriori} error estimator
\begin{align*}
\eta^2((\sigma_h,u_h),\mathcal M_h) & = \sum\limits_{K \in \mathcal M_h} h_K^2\Big(\|\delta_K(f - \dd ^-\sigma_h)\|_K^2 + \| f - \delta_K \dd u_h - \dd ^-\sigma_h\|_K^2 \Big)\\ 
& + \sum\limits_{e\in\mathcal F_h} h_e\Big(\|[\tr\star(f -  \dd ^- \sigma_h)]\|_e^2 
+ \|[\tr\star \dd u_h]\|_e^2 \Big), 
\end{align*}
which is just $\eta_0$ of Demlow and Hirani's estimators~\cite{Demlow2014}.
As we mentioned before, Demlow and Hirani~\cite{Demlow2014} uses the stability in $H\Lambda$ norm and their error estimator involves both the energy norm and the $L^2$-norm, c.f., \eqref{DH}. In contrast, we derive the error estimator using the orthogonality and thus obtain a tighter estimator for the energy norm only, i.e.,
\begin{equation}\label{CW}
\|\dd ^-(\sigma - \sigma_h)\| + \|\dd (u - u_h)\| \leq C\eta(\sigma_h, u_h).
\end{equation}
During the derivation of our error estimator, we develop discrete Poincar\'e inequalities and some new variations of regular decompositions of the differential forms which are of independent interest.

In this paper we only consider the domain $\Omega$ without harmonic forms, i.e., for contractable domains with trivial homology groups. If the topology of the domain is non-trivial, we can first approximate the harmonic forms, which is an eigenvalue problem with known (zero) eigenvalue, and then restrict to the space orthogonal to the harmonic forms. Of course, the harmonic form itself could be singular and the adaptive procedure can be also used.  On the convergence analysis of adaptive finite element  approximation of harmonic forms, we refer to the recent work by Demlow~\cite{Demlow2016}. Combination of the current work and Demlow's result in~\cite{Demlow2016} will be our future work.

To simplify the notation, we restrict ourselves to the homogenous Dirichlet boundary conditions. Our results are naturally extendable to the de Rham complex \eqref{deRham} corresponding to the Neumann boundary conditions. For mixed boundary conditions, establishing necessary properties of the Hodge decomposition, regular decomposition, and quasi-interpolant is more involved and has not been carried out in the literature to date~\cite{Costabel;McIntosh2010,GMM2011}.

In the marking step, besides $\eta(\sigma_h, u_h)$ we include an additional D\"orfler marking for $\eta(\sigma_h)$ which is crucial to obtain the quasi-orthogonality and the convergence proof without assuming the initial mesh size is sufficiently small. On the other hand, also due to this additional marking, we cannot prove the optimality of our adaptive algorithm following the standard approaches~\cite{Nochetto;Siebert;Veeser2009,Carstensen;Feischl;Praetoriu2014}. 

Another remark is that our results for differential forms $1\leq k \leq n-1$ are not consistent with the existing results~\cite{Chen;Holst;Xu2009,Huang;Xu2012,Holst;Mihalik;Ryan;2013} for $k=n$. When $k=n$, $\dd = 0$ in the short exact sequence \eqref{short}. Therefore $\|\dd ^{-}(\sigma - \sigma_h)\|$ can be bounded by the so-called data oscillation $\|h(f-f_h)\|$. A different orthogonality of $\|\sigma - \sigma_h\|$, i.e. in $L^2-$norm can be established and the convergence and optimality for $\|\sigma - \sigma_h\|$ (not involving the error $u-u_h$) is obtained consequently. In contrast, our results for $1\leq k \leq n-1$ are on the energy norm and the errors $\sigma - \sigma_h$ and $u - u_h$ are coupled together.

Throughout this paper, the notation $a \lesssim b$ (or $a \gtrsim b$) means there exists a positive constant $C$, which is independent of the mesh size $h$, may dependent on the order of the polynomial and may not be the same at different occurrences, such that $a \leq Cb$ (or $ a \geq Cb$). When the constant in an inequality plays a role in the convergence analysis, we will use $C_1, C_2, \ldots,$ with a specific index. 

The structure of this paper is as follows. In Section 2, we review basic definitions and properties of Hilbert complexes, the Hodge Laplacian equation, and regular decompositions and commuting quasi-interpolants. In Section 3, we obtain the orthogonality of $\sigma$ and the quasi-orthogonality of $u$. In Section 4, we establish \emph{a posteriori} upper bounds, and lower bounds for errors. In Section 5, we give the convergence of the adaptive algorithm, which is the main results of this paper. In Section 6, we present some examples and translate our results into standard vector calculus notations in three dimensions.

\section{Preliminaries}

In this section, following~\cite{Arnold.D;RS.Falk;Winther.R2006,Arnold.D;RS.Falk;Winther.R2010}, we will review basic definitions and properties of the de Rham complex and then introduce the mixed formulation of the Hodge Laplacian equation.
 
\subsection{de Rham Complexes and Poincar\'e Inequalities}

Let $\Omega \subset \mathbb R^n$, $n \geq 2$ be a bounded Lipschitz polyhedral domain. For an integer $0 \leq k \leq n$, $\Lambda^k(\Omega)$ represents the linear space of all smooth $k-$forms on $\Omega$. As $\Omega$ is a flat domain in $\mathbb R^n$, we can identify each tangent space of $\Omega$ with $\mbb R^n$. Given a $\omega\in \Lambda^k(\Omega)$ and vectors $v_1, \ldots, v_k$, we obtain a smooth map $\Omega \to \mbb R$ by $x\mapsto \omega_x(v_1, \ldots, v_k)$ for $x \in \mathbb R^n$. We use the default $n$-volume of $\mathbb R^n$ and denote the volume form in $\Lambda^n(\Omega)$ by $\dd {\rm v}$. We  define the $L^2$-inner product for two differential $k$-forms on $\Omega$ as
$$
\langle \omega, \mu \rangle = \int_{\Omega} \langle \omega_x, \mu_x \rangle \dd {\rm v}.
$$
We then take the completion of  $\Lambda^k(\Omega)$ in the corresponding norm to get  the Hilbert space $L^2\Lambda^k(\Omega)$. Similarly for any real number $r $, we can define the Sobolev spaces of $k-$forms   $H^r\Lambda^k(\Omega)$ with norm $\|\cdot\|_{H^r}$.

Given a smooth differential $k$-form $\omega \in \Lambda^k(\Omega)$, we define the exterior derivative 
\begin{align*}
\dd^k \omega_x(v_1,\cdots,v_{k+1}) & = \sum\limits_{j = 1}^{k+1}(-1)^{j+1}\partial_{v_j}\omega_x(v_1,\cdots,\hat v_j,\cdots,v_{k+1}),
\end{align*}
where the hat is used to indicate a suppressed argument.
Then $\dd^k: \Lambda^k(\Omega) \to \Lambda^{k+1}(\Omega)$ is a sequence of differential operators satisfying that the range of $\dd^k$ lies in the domain of $\dd^{k+1}$ i.e. $\dd^{k+1} \circ \dd^k  = 0$, for $k = 0,1,\cdots, n-1$. For the convenience of notation, later on we shall skip the superscript $k$ if there is no confusion.

The domain of the exterior derivative can be enlarged. 
Let 
$$
H\Lambda^k(\Omega) = \{\omega \in L^2\Lambda^k(\Omega)|\ \dd \omega \in L^2\Lambda^{k+1}(\Omega)\},
$$  
be the domain of $\dd $ with inner product $\langle \omega, \mu \rangle + \langle \dd \omega, \dd \mu \rangle$ and associated graph norm $\|\cdot\|_{H\Lambda}$. The de Rham complex
\begin{equation}\label{deRham}
\mathbb R\longrightarrow H\Lambda^0(\Omega) \stackrel{\dd}{\longrightarrow} H\Lambda^1(\Omega) \stackrel{\dd}{\longrightarrow} \cdots \stackrel{\dd}{\longrightarrow} H\Lambda^n(\Omega) \longrightarrow 0,     
\end{equation}
is then bounded in the sense that $\dd :\ H\Lambda^k(\Omega) \mapsto H\Lambda^{k+1}(\Omega)$ is a bounded linear operator.

Let $M$ be a smooth manifold. For any $x \in M$, we use $T_xM$ to denote the tangential space of $M$ at $x$. For any $k$-form $\omega \in \Lambda^k(M)$, we define $\tr\omega \in \Lambda^k(\partial M)$ satisfying
$$
\tr\omega(v_1,v_2,\cdots,v_k) = \omega(v_1,v_2,\cdots,v_k),
$$
for tangent vectors $v_i \in T_x\partial M  \subset T_x M\ (i = 1,2,\cdots,k)$.
The trace operator can be extended continuously to Lipschitz domain $\Omega$, and to $\tr:\ H^1\Lambda^k(\Omega) \mapsto H^{1/2}\Lambda^k(\partial \Omega)$ and $\tr:\ H\Lambda^k(\Omega)\mapsto H^{-1/2}\Lambda^k(\partial \Omega)$ (see~\cite[page 19]{Arnold.D;RS.Falk;Winther.R2006}). Define 
\begin{align*}
H_0\Lambda^k(\Omega) &= \{ \omega \in H\Lambda^k(\Omega):\tr \omega = 0 \text{ on } \partial \Omega\}\\
H_0^1\Lambda^k(\Omega) &= \{ \omega \in H^1\Lambda^k(\Omega): \tr \omega = 0 \text{ on } \partial \Omega\}
\end{align*}
For easy of presentation we will focus on the de Rham complex with homogenous trace
\begin{equation}\label{eq:complex}
0\longrightarrow H_0\Lambda^0(\Omega) \stackrel{\dd}{\longrightarrow} H_0\Lambda^1(\Omega) \stackrel{\dd}{\longrightarrow} \cdots \stackrel{\dd}{\longrightarrow} H_0\Lambda^n(\Omega) \longrightarrow 0.     
\end{equation}
Our results can be naturally extended to the de Rham complex \eqref{deRham} corresponding to the Neumann boundary conditions. For the mixed boundary conditions, establishing necessary properties of the Hodge decomposition, regular decomposition, and quasi-interpolant is more involved and has not been carried out in the literature to date~\cite{Costabel;McIntosh2010,GMM2011}.

Let $\wedge$ denote the wedge product and let $\star: \Lambda^k(\Omega) \mapsto \Lambda^{n-k}(\Omega)$ be the Hodge star operator. For any $\omega \in \Lambda^k(\Omega),  \mu \in \Lambda^{n-k}(\Omega)$, $\wedge$ and $\star$ are related by  
$$
\int_{\Omega} \omega \wedge \mu = \langle \star \omega, \mu \rangle,
$$
The co-derivative operator $\delta: \Lambda^k(\Omega) \mapsto \Lambda^{k-1}(\Omega)$ is defined as
$$
\delta \omega = (-1)^{k(n-k+1)}\star \, \dd \star \omega.
$$
In analogy with $H\Lambda^k(\Omega)$, we define the spaces
\begin{align*}
H^*\Lambda^k(\Omega) & =  \{ \omega \in L^2\Lambda^k(\Omega): \delta \omega \in L^2\Lambda^{k-1}(\Omega)\},\\
H_0^*\Lambda^k(\Omega) & =   \{ \omega \in H^*\Lambda^k(\Omega): \tr(\star \omega) = 0\}.
\end{align*}
Stokes' theorem implies 
\begin{equation*}
\langle \dd \omega,  \mu \rangle = \langle \omega,  \delta \mu \rangle + \int_{\partial \Omega}\tr \omega \wedge \tr\star \mu,\quad \omega \in \Lambda^{k-1}(\Omega),\ \mu\in \Lambda^k(\Omega).
\end{equation*}
Treat $\dd: H_0\Lambda^k\subset L^2\Lambda^k \to L^2\Lambda^{k+1}$ as a unbounded and densely defined operator. Then $\delta: H^*\Lambda^k(\Omega)\subset L^2\Lambda^{k+1}(\Omega) \mapsto L^2\Lambda^{k}(\Omega)$ is the adjoint of $\dd$ as 
\begin{equation}\label{eq:d_delta}
\langle \dd \omega,\mu \rangle  = \langle \omega ,\delta \mu \rangle, \quad \omega\in H_0\Lambda^k.
\end{equation}
Consequently we have a dual sequence of \eqref{eq:complex}
\begin{equation}\label{eq:complex_dual}
0\longleftarrow H^*\Lambda^0(\Omega) \stackrel{\delta}{\longleftarrow} H^*\Lambda^1(\Omega) \stackrel{\delta}{\longleftarrow} \cdots \stackrel{\delta}{\longleftarrow} H^*\Lambda^n(\Omega) \longleftarrow \mathbb R.     
\end{equation}

The kernel of $\dd$ in $H_0\Lambda^k$ can be decomposed as $\mathcal Z_0^k = \mathcal B_0^k \oplus^{\bot_{L^2}} \mathcal H^k_0$, where $\mathcal B_0^k$ is the range of $\dd$, i.e. $\mathcal B_0^k = \dd (H_0\Lambda^{k-1}(\Omega))$ and $\mathcal H_0^k$ is the space of the harmonic forms, i.e., $\mathcal H_0^k = \{ \omega \in H_0\Lambda^k(\Omega)\cap H^*\Lambda^k(\Omega):\ \dd \omega = 0 \text{ and } \delta\omega = 0 \}$.
The following Hodge decomposition has been established in~\cite[page 22]{Arnold.D;RS.Falk;Winther.R2006}
$$
L^2\Lambda^k(\Omega) = \mathcal B_0^k \oplus^{\bot_{L^2}} \mathcal H_0^k \oplus^{\bot_{L^2}} \delta H^*\Lambda^{k+1}(\Omega).
$$
Let $\mathcal K^k$ be the $L^2$ orthogonal complement of $\mathcal Z_0^k$ in $H_0\Lambda^k(\Omega)$, i.e., $\mathcal K^k =  H_0\Lambda^k(\Omega) \cap \delta H^*\Lambda^{k+1}(\Omega)$.
Then we have the Hodge decomposition of $H_0\Lambda^k(\Omega)$:
\begin{equation}\label{eq:Hodge_decom}
H_0\Lambda^k(\Omega) = \mathcal Z_0^k \oplus^{\bot_{L^2}} \mathcal K^k = \mathcal B_0^k \oplus^{\bot_{L^2}} \mathcal H_0^k \oplus^{\bot_{L^2}} \mathcal K^k.
\end{equation}

In this paper we consider the domain $\Omega$ without harmonic forms, namely we impose the following assumption

\medskip

\noindent ({\bf A}) We assume that $\Omega$ is simple in the sense that $\dim \mathcal H_0^k = 0$ for $k = 1, 2, \cdots, n-1.$

\medskip

If the topology of the domain is non-trivial, we can first approximate the harmonic forms, which is an eigenvalue problem with known (zero) eigenvalue. In practice the dimension of $\mathcal H^k$ is usually small. We may then consider the Hodge Laplacian equation on the space orthogonal to the harmonic forms. Of course, harmonic forms could be singular and an adaptive procedure can be used as well.
For adaptive finite element approximations of harmonic forms, we refer to the recent work by Demlow~\cite{Demlow2016}.

In the rest part of this paper, when spaces of the consecutive differential forms are involved, we use the short sequences 
\begin{equation}\label{eq:complex_Hilbert}
H_0\Lambda^-(\Omega) \stackrel{\dd ^-}{\longrightarrow} H_0\Lambda(\Omega) \stackrel{\dd}{\longrightarrow} H_0\Lambda^+(\Omega) 
\end{equation}
or the one with the Hodge decomposition
\begin{equation}\label{eq:ZK}
\mathcal Z_0^- \oplus ^{\bot_{L^2}} \mathcal K^-  \stackrel{\dd ^-}{\longrightarrow} \mathcal Z_0 \oplus ^{\bot_{L^2}} \mathcal K \stackrel{\dd}{\longrightarrow} \mathcal Z_0^+ \oplus ^{\bot_{L^2}} \mathcal K^+.
\end{equation}

We have the following embedding result of the space $H_0\Lambda^k(\Omega) \cap H^*\Lambda^k(\Omega)$ on Lipschitz polyhedra domain $\Omega$. 
\begin{lemma}[Regularity of Intersection of Spaces~(\cite{Mitrea;Mitrea;Taylor2001}, Theorem 11.2)] \label{lem:regularity}
Assume $\Omega$ is a Lipschitz polyhedra. Then there exists a constant $\frac{1}{2} \leq s \leq 1$ depending only on $\Omega$ such that $H_0\Lambda^k(\Omega) \cap H^*\Lambda^k(\Omega)\hookrightarrow H^s\Lambda^k(\Omega)$, i.e., there exists a constant $C_r > 0$ depending only on $\Omega$, such that for all $v \in H_0\Lambda^k(\Omega) \cap H^*\Lambda^k(\Omega)$
$$
\|v\|_{H^s} \leq C_r \left( \|v\| + \|\dd v\| + \|\delta v\|  \right).
$$
\end{lemma}

As $H^s\Lambda^k(\Omega), s\in [1/2, 1]$, is compactly embedded into $L^2\Lambda^k(\Omega)$, using the standard compactness argument, we  can  get 
the following Poincar\'e inequality~\cite[Page 23, Theorem 2.2]{Arnold.D;RS.Falk;Winther.R2006} which will play an important role in our analysis.
\begin{lemma}[Poincar\'e Inequality]\label{lem:Poincare}
There exists a constant $C_{pc}>0$, such that
\begin{equation}\label{eq:Poincare}
\| \omega\| \leq C_{pc}\left( \|\dd \omega\| + \|\delta \omega\|  \right),
\end{equation}
for $\omega \in H_0\Lambda^k(\Omega) \cap H^*\Lambda^k(\Omega)$.
\end{lemma}

By the definition of $\mathcal K$,  $\mathcal K =  H_0\Lambda(\Omega) \cap \delta H^*\Lambda^+(\Omega)$, we have $\delta  \mathcal K = 0$ and thus
\begin{equation}
\label{eq:Poincare_d}  \| \omega \|  \leq  C_{pc}\|\dd \omega\|,\qquad  \omega \in \mathcal K.
\end{equation}
Poincar\'e inequality \eqref{eq:Poincare_d} implies that $\mathcal K$ is a Hilbert space equipped with the inner product $\langle \dd(\cdot), \dd(\cdot) \rangle$.

\subsection{The Hodge Laplacian Equation}
The Hodge Laplacian problem reads as: given $f\in L^2\Lambda(\Omega)$, find $u\in H_0\Lambda^k(\Omega) \cap H^*\Lambda^k(\Omega)$ such that
\begin{equation}
\label{eq:Hodge_Lap}
\mathcal L u  = f,
\end{equation}
where $\mathcal L = \dd ^- \delta + \delta^+ \dd$ is called the Hodge Laplacian. 

A direct discretization of \eqref{eq:Hodge_Lap} would require a conforming discretization of the intersection space $H_0\Lambda^k(\Omega) \cap H^*\Lambda^k(\Omega)$. Using Lagrange finite element spaces which are subspaces of $H^1\Lambda^k$ for \eqref{eq:Hodge_Lap} are problematic~\cite{Arnold.D;RS.Falk;Winther.R2010}. For example, when the domain is non-convex or non-smooth, approximations using continuous piecewise linear functions will converge but not converge to the true solution. It might be explained by the  embedding  $H_0\Lambda^k(\Omega) \cap H^*\Lambda^k(\Omega)\hookrightarrow H^s\Lambda^k(\Omega)$. When the domain is non-convex or non-smooth, $s<1$ but the Lagrange finite element spaces will provide approximations converging in $H^1\Lambda^k$.

Introduce a new variable $\sigma = \delta u$. The mixed formulation of the Hodge Laplacian problem is: Given $f \in L^2\Lambda(\Omega)$, find $(\sigma,u) \in H_0\Lambda^{-}(\Omega) \times H_0\Lambda(\Omega)$ such that
\begin{equation}
\label{eq:mixed_formulation_con}
\left\{
\begin{array}{llrlll}
\langle\sigma,\tau\rangle & - & \langle \dd ^- \tau, u\rangle & = & 0 & \text{for all }\tau \in H_0\Lambda^{-}(\Omega),\\
\langle \dd ^- \sigma,v\rangle & + & \langle \dd u , \dd v\rangle & = & \langle f, v \rangle & \text{for all } v \in H_0\Lambda(\Omega).
\end{array}   
\right.
\end{equation}
The saddle point system \eqref{eq:mixed_formulation_con} is well-posed; see \cite[Page 83, Section 7.5]{Arnold.D;RS.Falk;Winther.R2006}. Let $(\sigma,u)$ be the solution of equation \eqref{eq:mixed_formulation_con}.
 Although $u\in H_0\Lambda(\Omega)$ only, the first equation implies $\sigma = \delta u$ holds in $L^2$ and thus $u\in H_0\Lambda(\Omega)\cap H^*\Lambda(\Omega)$. Writing $\sigma = \delta u$ and $\langle \dd ^- \sigma,v\rangle = \langle \sigma, \delta v\rangle$, we can eliminate $\sigma$ from equation \eqref{eq:mixed_formulation_con} and get
$$
\langle \delta u,\delta v\rangle + \langle \dd u , \dd v\rangle  =  \langle f, v \rangle \quad \forall v\in H_0\Lambda(\Omega) \cap H^*\Lambda(\Omega),
$$ 
which is the primary weak formulation of \eqref{eq:Hodge_Lap}.

\subsection{Mixed Finite Element Methods of the Hodge Laplacian Equation} 
Let $\mathcal T_h$ be a shape-regular triangulation of $\Omega$. For an $n$-simplex $K \in\mathcal T_h$, an $(n-1)$-dimensional face of $K$ is an $(n-1)$-simplex generated by $n$ vertexes of $K$.  Let $\mathcal E_h$ be the set of all interior $(n-1)$-dimensional faces in $\mathcal T_h$. For each $n$-simplex $K\in \mathcal T_h$, we define $h_K = |K|^{1/n}$ and 
$h = \max\limits_{K \in \mathcal T_h} \{h_K\}$, where $|K|$ is the $n$-dimensional volume.
For each $(n-1)$-simplex $e \in \mathcal E_h$, we define $h_e = |e|^{1/(n-1)}$ with $|e|$ the $(n-1)$-dimensional volume. 

Let $\mathcal P_r(\mathbb R^n)$ denote the space of polynomials in $n$ variables of degree at most $r$ and $\mathcal H_r(\mathbb R^n)$ of homogeneous polynomial functions of degree $r$. We can then define spaces of polynomial differential forms $\mathcal P_r\Lambda^k(\mathbb R^n)$ and $\mathcal H_r\Lambda^k(\mathbb R^n)$ by using corresponding polynomials as the coefficients. To simplify the notation, we will suppress $\mathbb R^n$ from the notation.
For each integer $r\geq n$, we have the polynomial sub-complex of the de Rham complex
\begin{equation}\label{Pexact}
0\stackrel{}{\longrightarrow}\mathcal  P_r\Lambda^0 \stackrel{\dd }{\longrightarrow} \mathcal P_{r-1}\Lambda^1 \stackrel{\dd}{\longrightarrow} \cdots \stackrel{\dd }{\longrightarrow} \mathcal P_{r-n}\Lambda^n \stackrel{}{\longrightarrow} 0.
\end{equation}
Given a point $x\in \mathbb R^n$, there is a natural identification of $x$ to a vector in $\mathbb R^n$ which can be also identified with a vector in the tangent space $T_x\mathbb R^n$. Thus treat $x$ as a vector in $T_x\mathbb R^n$ and define the Koszul operator $\kappa: \Lambda^k(\mathbb R^n)$ to $\Lambda^{k-1}(\mathbb R^n)$ by the formula
$$
(\kappa\omega)_x(v_1,v_2,\cdots,v_{k-1}) =  \omega_x(x,v_1,\cdots,v_{k-1}).
$$
On the space $\mathcal H_r\Lambda^k$, such $\kappa$ satisfying the identity $\kappa \dd + \dd \kappa = (k+r){\rm id}$ \cite[Theorem 3.1]{Arnold.D;RS.Falk;Winther.R2006} and there is a direct sum 
$$
\mathcal H_r\Lambda^k = \kappa\mathcal H_{r-1}\Lambda^{k+1} \oplus \dd \mathcal H_{r+1}\Lambda^{k-1}.
$$ 
Based on this decomposition, we can introduce the incomplete polynomial differential form 
$$
\mathcal P_r^-\Lambda^k = \mathcal P_{r-1}\Lambda^k + \kappa \mathcal H_{r-1}\Lambda^{k+1}
$$ 
and, for $r\geq 1$, obtain the following sub-complex of the de Rham complex 
$$
0\stackrel{}{\longrightarrow}\mathcal  P^-_r\Lambda^0 \stackrel{\dd }{\longrightarrow} \mathcal P^-_{r}\Lambda^1 \stackrel{\dd}{\longrightarrow} \cdots \stackrel{\dd }{\longrightarrow} \mathcal P^-_{r}\Lambda^n \stackrel{}{\longrightarrow} 0.
$$

For each simplex $K\in \mathcal T_h$, we denoted by $\mathcal P_r\Lambda^k(K)$ or $\mathcal P_r^-\Lambda^k(K)$, the spaces of forms obtained by restricting the forms in $\mathcal P_r\Lambda^k(\mathbb R^n)$ and $\mathcal P_r^-\Lambda^k(\mathbb R^n)$, respectively, to $K$. 
We then introduce the finite element spaces
\begin{align*}
\mathcal P_r\Lambda^k(\mathcal T_h) &= \{\omega\in H\Lambda^k(\Omega)|\ \omega|_K \in \mathcal P_r\Lambda^k(K),\ K \in \mathcal T_h\},\\
\mathcal P_r^-\Lambda^k(\mathcal T_h) &= \{\omega\in H\Lambda^k(\Omega)|\ \omega|_K \in \mathcal P_r^-\Lambda^k(K),\ K \in \mathcal T_h\}.
\end{align*}
For each $k$, we choose $V_h^k = \mathcal P_r\Lambda^k(\mathcal T_h) \cap H_0\Lambda^k(\Omega)$ or $V_h^k = \mathcal P_r^-\Lambda^k(\mathcal T_h)\cap H_0\Lambda^k(\Omega)$ so that $(V_h^k, \dd^k)$ forms a sub-complex of $(H_0\Lambda^k(\Omega), \dd^k)$. Again for the consecutive spaces, we shall use the short sequence $V_h^{-} \stackrel{\dd ^-}{\longrightarrow}  V_h\stackrel{\dd}{\longrightarrow}  V_h^+$.

We define the discrete co-derivative $\delta_h:V_h \mapsto V_h^{-}$ as the $L^2$-adjoint of $\dd ^{-}: V_h^{-} \to V_h$, i.e. given $w_h \in V_h$, $\delta_h w_h$ is the unique element in $V_h^{-}$ such that
\begin{equation}\label{eq:def_delta_h}
\langle \delta_h w_h , v_h\rangle = \langle w_h, \dd ^{-} v_h \rangle, \quad \text{for all } v_h \in V_h^{-}.
\end{equation} 
To compute $\delta_h w_h$, we need to invert the so-called mass matrix and thus in general $\delta_h w_h$ is a non-local operator. Note that in the finite dimensional space setting, $\delta_h$ is a bounded linear operator although its norm is not uniform to $h$.
Recall that the adjoint $\delta$ satisfies $\langle \delta w_h,v \rangle = \langle w_h,\dd ^{-} v \rangle,\ \text{for all }v \in H_0\Lambda^{-}(\Omega)$. For the discrete co-derivative $\delta_h w_h$ such relation holds only for a subspace $V_h \subset H_0\Lambda(\Omega)$. Therefore, the space $V_h$ is a conforming discretization of $H_0\Lambda(\Omega)$ but not a conforming discretization of $H^*\Lambda^k(\Omega)$.

The discrete Hodge decomposition of $V_h$ is
\begin{equation}
\label{eq:Hodge_decom_dis}
V_h = \mathcal Z_{0,h} \oplus^{\bot_{L^2}} \mathcal K_h,
\end{equation}
where $\mathcal Z_{0,h} = \ker(\dd) \cap V_h \subset \mathcal Z_0$ and $\mathcal K_h$ is the $L^2$ orthogonal complement of $\mathcal Z_{0,h}$ in $V_h$. Since we consider the Hodge Laplacian without harmonic forms, $\mathcal Z_{0,h}= \dd V_h^{-}$ and equation \eqref{eq:def_delta_h} implies that $\mathcal K_h$ is the kernel of $\delta_h$. Generally $\mathcal Z_{0,h}$ is a proper subspace of $\mathcal Z_0$, but $\mathcal K_h \not\subset \mathcal K$. For instance, when $V_h\subset H_0(\curl; \Omega)$ is the N\'{e}d\'{e}lec edge element space \cite{Nedelec.J1980,Nedelec.J1986}, the function in $\mathcal K_h$ is called discrete divergence free which is in general not divergence free in the continuous level.

We have the following discrete Poincar\'e inequality, cf. 
\cite[Theorem 5.11]{Arnold.D;RS.Falk;Winther.R2006}. 
\begin{lemma}[Discrete Poincar\'e Inequality for $\dd$]
\label{lem:Poincare_Gen}
There exists a constant $C_p > 0$  independent of $h$ such that
\begin{equation}\label{eq:Poincare_K_h}
\|w_h\| \leq C_p \|\dd w_h\|,\qquad w_h \in \mathcal K_h.
\end{equation}
\end{lemma}
\smallskip

As the adjoint operator of $\dd ^-: V_h^- \to V_h$, we have the following discrete Poincar\'e inequality for $\delta_h$ as well. 
\begin{lemma}[Discrete Poincar\'e Inequality for $\delta_h$]\label{lem:Poincare_dis}
Let $C_p$ be the constant in \eqref{eq:Poincare_K_h}. Then we have
\begin{equation}
\label{eq:Poincare_dis}
\|v_h\| \leq C_p \|\delta_{h} v_h\|, \qquad v_h \in \mathcal Z_{0,h}.
\end{equation}
\end{lemma}
\begin{proof}
For any $v_h \in \mathcal Z_{0,h}$, there exists $\rho_h \in \mathcal K_h^{-}$, such that $v_h = \dd ^- \rho_h$. By Poincar\'e inequality \eqref{eq:Poincare_K_h}, we have
$$
 \|\rho_h\| \leq C_p \|\dd ^- \rho_h\| = C_p \|v_h\|.
$$
Then
$$
\|v_h\|^2 = \langle v_h, \dd ^-\rho_h \rangle = \langle \delta_h v_h, \rho_h \rangle \leq \|\delta_h v_h\|\|\rho_h\|  \leq C_p  \|\delta_h v_h\| \|v_h\|.
$$
The desired result then follows.
\end{proof}
 
Combination the  Hodge decomposition and the above  discrete Poincar\'e inequalities, we obtain the following discrete Poincar\'e inequality on the finite element space $V_h$ which is analogue to the continuous  one; c.f.  Lemma \ref{lem:Poincare}.
\begin{theorem}[Discrete Poincar\'e Inequality]\label{th:discretePoincare}
There exists a constant $C_p > 0$  independent of $h$ such that
\begin{equation}\label{discretePoinare}
 \| v_h \|^2 \leq C_p^2\left (\|\dd v_h \|^2 + \|\delta_h v_h\|^2 \right ),\qquad v_h \in V_h.
\end{equation}
\end{theorem}
\begin{proof}
For any $v_h\in V_h$, by the discrete Hodge decomposition~\eqref{eq:Hodge_decom_dis}, there exist $z_h \in \mathcal Z_{0,h}$ and $\omega_h \in \mathcal K_h$, such that
$$
v_h = z_h \oplus^{\bot_{L^2}} \omega_h\qquad\text{and}\qquad \delta_h v_h = \delta_hz_h,\ \dd v_h = \dd \omega_h.
$$
Lemmas \ref{lem:Poincare_dis} and \ref{lem:Poincare_Gen} imply that
$$
\|z_h\| \leq C_p \|\delta_hz_h\| = C_p\|\delta_h v_h\|,\qquad\|\omega_h\|\leq C_p\|\dd \omega_h\| = C_p\|\dd v_h\|.
$$
Take square and sum together to get the desired inequality.
\end{proof}

These Poincar\'e inequalities imply the isomorphisms of the discrete spaces
$$
\mathcal K_h^- \mathrel{\mathop{\rightleftarrows}^{\dd ^-}_{\delta_h}}   \mathcal Z_{0,h},   \quad
\mathcal K_h \mathrel{\mathop{\rightleftarrows}^{\dd}_{\delta_h^+}}   \mathcal Z_{0,h}^+.
$$
 
The mixed finite element method for \eqref{eq:mixed_formulation_con} is constructed as follows: Given $f\in L^2\Lambda^k$, find $(\sigma_h,u_h) \in V_h^{-} \times V_h$ such that
\begin{equation}
\label{eq:mixed_formulation_dis}
\left\{
\begin{array}{llrlll}
\langle\sigma_h,\tau_h\rangle & - & \langle \dd ^- \tau_h, u_h\rangle & = & 0 & \text{for all }\tau_h \in V_h^{-},\\
\langle \dd ^- \sigma_h,v_h\rangle & + & \langle \dd u_h , \dd v_h\rangle & = & \langle f, v_h \rangle & \text{for all } v_h \in V_h.
\end{array}   
\right.
\end{equation}
By definition, we can write the first equation as $\sigma_h = \delta_h u_h$ and the second equation as
\begin{equation}\label{eq:discreteprimary}
\langle \delta_h u_h, \delta_h v_h \rangle + \langle \dd u_h, \dd v_h\rangle  =  \langle f, v_h \rangle, \quad \text{for all } v_h \in V_h.
\end{equation}
The well-posedness of equation \eqref{eq:discreteprimary} follows from the discrete Poincar\'e inequalities, c.f., Theorem \ref{th:discretePoincare}.

We will consider the convergence of the adaptive finite element methods for the Hodge Laplacian equation. When $k = 0$, $H_0\Lambda^0(\Omega) = H_0^1(\Omega)$, the Hodge Laplacian equation is the scalar Poisson equation, for which the convergence theory of AFEM has been well-studied. When $k = n$, $H_0\Lambda^n(\Omega) = L_0^2(\Omega)$, which is the subspace of $L^2$ integrable functions with vanishing means, the Hodge Laplacian equation is the mixed formulation of Poisson equation for which the convergence of AMFEM can be found in e.g.~\cite{Chen;Holst;Xu2009,Huang;Xu2012}. When $k=n$, $\dd = 0$ in the short exact sequence \eqref{short}. Therefore $\|\dd ^{-}(\sigma - \sigma_h)\|$ can be bound by the so-called data oscillation $\|h(f-f_h)\|$. An orthogonality of $\|\sigma - \sigma_h\|$ can be established and the convergence and optimality for $\|\sigma - \sigma_h\|$, which is now decoupled with the error $u-u_h$, is obtained consequently; see~\cite{Chen;Holst;Xu2009,Huang;Xu2012,Holst;Mihalik;Ryan;2013}. In this paper, we focus on the case $1\leq k \leq n-1$. Because of different norms used, the results in this paper are not consistent with these works for $k = n$.

\subsection{Regular Decompositions}
The Hodge decomposition \eqref{eq:Hodge_decom} is $L^2$-orthogonal but each component is merely in $L^2$. To be able to define interpolations to the finite element spaces, we need the following regular decomposition in which each component now is in $H^1$.
\begin{lemma}[Regular Decomposition~\cite{Kress1971,Girault1986,Amrouche1998,Hiptmair2002,Hiptmair;Xu2007,Costabel;McIntosh2010,Hiptmair;Li;Zou2012,Demlow2014}]
\label{lem:regular_decom}
Assume that $\Omega$ is a bounded Lipschitz domain in $\mathbb R^n$. For any integer $1 \leq k \leq n-1$, and any $v \in H_0\Lambda^k(\Omega)$, there exist $\varphi \in H_0^1\Lambda^{k-1}(\Omega)$ and $\phi \in H_0^1\Lambda^k(\Omega)$ such that $ v = \dd\varphi + \phi$ and
\begin{equation}\label{eq:regular_decom}
\|\varphi\|_{H^1} + \|\phi\|_{H^1} \lesssim \|v\|_{H\Lambda}.
\end{equation}
\end{lemma}

A combination of Hodge decomposition \eqref{eq:Hodge_decom} and regular decomposition \eqref{eq:regular_decom} leads to the following decomposition.
\begin{lemma}[Decomposition]\label{lem:Hodge_Regular_dec}
For any $v \in H_0\Lambda^k(\Omega)$ $(1 \leq k \leq n - 1)$, there exist $\rho \in H_0\Lambda^{k-1}(\Omega)$, $\omega\in H_0\Lambda^k(\Omega)$, $\varphi \in H_0^1\Lambda^{k-1}(\Omega)$ and $\phi \in H_0^1\Lambda^k(\Omega)$ such that
$$
v = \dd \rho \oplus^{\bot_{L^2}} \omega,\qquad \omega = \dd \varphi + \phi,
$$
and
$$
\|\varphi \|_{H^1} + \|\phi\|_{H^1} \lesssim \|\dd v\|.
$$
\end{lemma}
\begin{proof}
For any $v \in H_0\Lambda^k(\Omega)$, Hodge decomposition \eqref{eq:Hodge_decom} implies that there exist $\rho \in H_0\Lambda^{k-1}(\Omega)$ and $\omega\in H_0\Lambda^k(\Omega)$, such that
$$
v = \dd \rho \oplus^{\bot_{L^2}} \omega\quad \text{ and thus }\quad \dd \omega = \dd v.
$$
Then the regular decomposition \eqref{eq:regular_decom} of $\omega$ leads to the existence of $\varphi \in H_0^1\Lambda^{k-1}(\Omega)$ and $\phi \in H_0^1\Lambda^k(\Omega)$ such that
$$
\omega = \dd \varphi + \phi\quad \text{ and }\quad \|\varphi\|_{H^1} + \|\phi\|_{H^1} \lesssim \|\omega\|_{H\Lambda}.
$$
Furthermore use Poincar\'e inequality \eqref{eq:Poincare_d} for $\omega \in \mathcal K$, we have
$$
 \|\varphi\|_{H^1} + \|\phi\|_{H^1} \lesssim \|\omega\|_{H\Lambda} \lesssim \|\dd \omega\| = \|\dd v\|.
$$
\end{proof}

\subsection{Commuting Quasi-interpolations}
We review several interpolations connecting the differential forms in the  continuous and  discrete levels. The most important properties are the commuting  diagram shown in the figure below.

\begin{diagram}
\cdots & \rTo^{\dd}  & \Lambda^k(\Omega)    & \rTo^{\dd}  &  \Lambda^{k+1}(\Omega) & \rTo^{\dd}  & \cdots \\
&  & \dTo^{\Pi_h^k}              &                  &  \dTo^{\Pi_h^{k+1}} &  &  \\
\cdots & \rTo^{\dd}  &  V_h^k                          & \rTo^{\dd}   &  V_h^{k+1} &\rTo^{\dd}  & \cdots
\end{diagram}

The canonical interpolations, i.e., interpolations  using degree of freedoms in the construction of finite element spaces, requires higher smoothness than merely in $H\Lambda^k$. For example, for Lagrange elements, the nodal values only exists for function in $H^{k}, k> n/2\geq 1$. Quasi-interpolations using average in a neighborhood of a vertex and still remain the commuting property can be found in Sch\"oberl \cite{Schoberl.J2001}. We shall use the following variants adapted to the homogenous boundary condition.

\begin{lemma}[Quasi-Interpolations~\cite{Schoberl2007,Demlow2014}]
\label{lem:quasi_inter}
There exists a sequence of operators $\Pi_h^k: H_0\Lambda^k(\Omega)\mapsto V_h^k$ for $1 \leq k \leq n-1$ satisfying 
$$
\dd^k \ \Pi_h^k = \Pi_h^{k+1} \dd^k ,
$$ 
and
for any $v \in H_0\Lambda^k(\Omega)$,  with $\|v\|_{H\Lambda} \leq 1$, there exist $\varphi \in H_0^1\Lambda^{k-1}(\Omega)$ and $z \in H_0^1\Lambda^k(\Omega)$ such that
$$
v  = \dd \varphi + z,\qquad  \qquad \Pi_h v = \dd \Pi_h \varphi + \Pi_h z,
$$
and
\begin{align}\label{eq:quasi-inter}
\sum\limits_{K \in \mathcal T_h}& \Big[  h_K^{-2}\left(\|\varphi - \Pi_h\varphi\|_K^2 + \|z - \Pi_h z\|_K^2 \right) \\
&  + h_K^{-1} \left(  \|\tr(\varphi - \Pi_h \varphi)\|_{\partial K}^2 + \|\tr(z - \Pi_h z)\|_{\partial K}^2 \right) \Big] \lesssim 1.  \nonumber
\end{align}
\end{lemma}

Again we will suppress the subscript $k$ and write as $\Pi_h$ if the index is clear from the context.
With a simple scaling and rewriting $\varphi - \Pi_h \varphi$ and $z - \Pi_h z$ in Lemma \ref{lem:quasi_inter} as $\varphi$ and $z$, respectively, we get the following localized Hodge decomposition of the interpolation error which generalizes a similar result for Maxwell's equations~\cite{Schoberl2007}. Note that the trace of a differential form is in general in $H^{-1/2}$ space which is not computable. To this end, we introduce the piecewise $H^1$ space
$$
H^1\Lambda^k(\mathcal T_h) := \left\{g \in L^2\Lambda^k(\Omega):\ g|_K \in H^1\Lambda^k(K)\text{ for all } K \in \mathcal T_h  \right\}.
$$

\smallskip
\begin{lemma}[Localized Hodge decomposition of Interpolation Error]\label{lem:quasi_inter_1}
For any integer $1\leq k \leq n-1$, and any $v \in H_0\Lambda^k(\Omega)$, there exist $\varphi \in H_0\Lambda^{k-1}(\Omega)\cap H^1\Lambda^{k-1}(\mathcal T_h)$ and $z \in H_0\Lambda^k(\Omega)\cap H^1\Lambda^k(\mathcal T_h)$ such that
$$
v - \Pi_h v = \dd \varphi + z,
$$
and
\begin{align}
\label{eq:iner_k_1}
\sum\limits_{K \in \mathcal T_h}\left( h_K^{-2}\|\varphi \|_K^2 + h_K^{-1}\|\tr(\varphi)\|_{\partial K}^2   \right) & \lesssim  \|v\|_{H\Lambda}^2,\\
\label{eq:inter_k_2}\sum\limits_{K \in \mathcal T_h}\left( h_K^{-2}\|z\|_K^2 + h_K^{-1}\|\tr(z)\|_{\partial K}^2   \right) & \lesssim  \|v\|_{H\Lambda}^2.
\end{align} 
\end{lemma}

Operator $\Pi_h$ is $L^2$ stable, commuting with the exterior derivative, and has the local approximation property. It is, however, not a projection. Namely $\Pi_hv_h \neq v_h$ for $v_h\in V_h$. 
In \cite[page 65-67]{Arnold.D;RS.Falk;Winther.R2006} and \cite{Christiansen;Winther2008}, the authors constructed a $L^2$ bounded projection operator $\tilde\Pi_h$. The idea is to compose $\Pi_h$ with a right inverse from $V_h \to L^2\Lambda$. The resulting operator is still commuting with the exterior derivative. We collect the desirable properties of $\tilde\Pi_h$ below:
\begin{equation}\label{AFWprojection}
\|\tilde\Pi_h \| \lesssim 1, \quad \dd \tilde\Pi_h  = \tilde\Pi_h \dd, \quad \tilde\Pi_h|_{V_h} = {\rm id}.
\end{equation}
From that we can easily derive the error estimate
\begin{equation} \label{eq:tilde_Pi}
\|(I - \tilde\Pi_h) v \| \lesssim h^s \| v \|_{H^s},\quad v\in H^s\Lambda, s\in [0,1].
\end{equation}

\subsection{Norms of Spaces}  Let $(\sigma,u)$ and $(\sigma_h,u_h)$ be the solutions of equations \eqref{eq:mixed_formulation_con} and \eqref{eq:mixed_formulation_dis}, respectively. It is easy to see that $\sigma = \delta u$ and $\sigma_h = \delta_h u_h$. 
We define subspaces consisting of the graph of $\delta$ and $\delta_h$:
\begin{align*}
\mathcal V & = \{(\delta v,v) :\ v \in H_0\Lambda(\Omega) \text{  and }\delta v\in H_0\Lambda^{-}(\Omega) \}\subset  H_0\Lambda^{-}(\Omega) \times H_0\Lambda(\Omega) ,\\
\mathcal V_h& = \{(\delta_h v_h,v_h):\ v_h \in V_h \} \subset  V_h^{-} \times V_h.
\end{align*} 
Note that $V_h^{-} \times V_h \subset H_0\Lambda^{-}(\Omega) \times H_0\Lambda(\Omega)$, but $\mathcal V_h \not\subset \mathcal V$ since $\delta_h$ is not a conforming discretization of $\delta$.  This non-nestedness is the main difficulty to prove the convergence of the mixed method.

For any $(\tau,v) \in \mathcal V$, define
\begin{align}\label{eq:norm1}
\|(\tau,v)\| & = \left( \|\dd ^-\tau\|^2 + \|\dd v\|^2  \right)^{1/2}.
\end{align}
It is obvious that $\|(\cdot,\cdot)\|$ is not a norm on the space $H_0\Lambda^{-}(\Omega) \times H_0\Lambda(\Omega)$. But we declare that $\|(\cdot,\cdot)\|$ is a norm on the subspaces $\mathcal V$ and $\mathcal V_h$.
\begin{lemma}[Norms]
\label{lem:norm}
Equation \eqref{eq:norm1} defines norms for both the spaces $ \mathcal V$ and $\mathcal V_h$.
\end{lemma}
\begin{proof}
For any $(\tau,v) \in \mathcal V$,
it suffices to verify the positivity of $\|(\tau,v)\|$. If $\|(\tau,v)\| = 0$, then $\dd \tau = 0$ and $\dd v = 0$. The fact that $\tau = \delta v$ and Poincar\'e inequality \eqref{eq:Poincare_d} imply 
$$
\|\delta v\| = \|\tau\| \lesssim \|\dd \tau\| = 0,
$$
which means $\delta v = \tau = 0$. Since there is no harmonic form, $\dd v = \delta v = 0$ implies $v = 0$.

Verification on the space $\mathcal V_h$ is similar by using discrete Poincar\'e inequality \eqref{discretePoinare}.
\end{proof}

\section{Quasi-Orthogonality}

We shall consider two conforming triangulations $\mathcal T_h$ and $\mathcal T_H$ which are nested in the sense that $\mathcal T_h$ is a refinement of $\mathcal T_H$. Therefore, the corresponding finite element spaces are nested, i.e., $V_H \subset V_h \subset H_0\Lambda(\Omega)$. Let $(\sigma,u)\in \mathcal V$ be the solution of equation \eqref{eq:mixed_formulation_con}, and $(\sigma_h,u_h) \in \mathcal V_h$ and $(\sigma_H,u_H)\in \mathcal V_H$ be the solutions of the equation \eqref{eq:mixed_formulation_dis} on the meshes $\mathcal T_h$ and $\mathcal T_H$, respectively. 
Due to the non-nestedness of $\mathcal V,\ \mathcal V_h$ and $\mathcal V_H$, $(\sigma_h,u_h)$ (or $(\sigma_H,u_H)$) is not an orthogonal projection of $(\sigma,u)$ from $\mathcal V$ to $\mathcal V_h$ (or $\mathcal V_H$). The lack of orthogonality is the main difficulty which complicates the convergence analysis of mixed finite element methods. In this section, we shall reveal some full and partial orthogonality implicitly contained in the mixed formulation.

We first explore an orthogonality of the error $\dd ^-(\sigma - \sigma_h)$.
By equations \eqref{eq:mixed_formulation_con} and \eqref{eq:mixed_formulation_dis}, and the fact that $\dd ^- V_h^{-} = \mathcal Z_{0,h} \subset V_h \subset H_0\Lambda(\Omega)$, we have the following orthogonality.
\begin{lemma}[Orthogonality of $\sigma$]
\label{lem:sigma_ortho}
Let  $(\sigma,u)\in H_0\Lambda^{-}(\Omega)\times H_0\Lambda(\Omega)$ be the solution of equation \eqref{eq:mixed_formulation_con}  and $(\sigma_h,u_h) \in V_h^{-} \times V_h$ be the solutions of equation \eqref{eq:mixed_formulation_dis} on meshes $\mathcal T_h$. It holds
\begin{equation}\label{eq:or_sigma}
\langle \dd ^-(\sigma - \sigma_h), \dd ^-\tau_h \rangle = 0\qquad \text{for all }\tau_h \in V_h^{-},
\end{equation}
and consequently
\begin{equation}\label{eq:orth_sigma}
\|\dd ^-(\sigma - \sigma_h)\|^2 = \|\dd ^-(\sigma - \sigma_H)\|^2 - \|\dd ^-(\sigma_h - \sigma_H)\|^2.
\end{equation}
\end{lemma}
\begin{proof}
For any $\tau_h \in V_h^{-}$, take $v = \dd ^-\tau_h$ in the second equation of \eqref{eq:mixed_formulation_con}, and $v_h = \dd ^-\tau_h$ in the second equation of \eqref{eq:mixed_formulation_dis}. Then we get
\begin{align*}
\langle \dd ^-\sigma,\dd ^-\tau_h \rangle  = \langle f,\dd ^-\tau_h \rangle \quad \text{  and  }\quad
\langle \dd ^-\sigma_h,\dd ^-\tau_h \rangle = \langle f,\dd ^-\tau_h \rangle.
\end{align*}
Thus \eqref{eq:or_sigma} follows by subtraction, and \eqref{eq:orth_sigma} is a consequence of \eqref{eq:or_sigma}. 
\end{proof}

We then use a relation between $\mathcal K_h$ and $\mathcal K$ to obtain a quasi-orthogonality of the error $u - u_h$. 
We define $Q_{\mathcal K}: \mathcal K_h \mapsto \mathcal K$ as the $L^2$ projection, i.e., for a $v_h \in \mathcal K_h,  Q_{\mathcal K} v_h\in \mathcal K$ satisfying  $\langle Q_{\mathcal K} v_h, \phi \rangle = \langle v_h, \phi \rangle, \forall  \phi \in \mathcal K.$ Note that $Q_{\mathcal K} v_h = v_h - \dd ^-\psi$ where $\psi \in \mathcal K^{-}$ is determined uniquely by equation $\langle \dd ^- \psi, \dd ^-\varphi \rangle = \langle v_h, \dd ^-\varphi \rangle$ for all $\varphi \in \mathcal K^{-}$. Therefore $\dd ^- Q_{\mathcal K} v_h = \dd ^- v_h$. Furthermore, we have the following error estimate, generalizing results for $H(\curl)$ and $H(\div)$ spaces developed in~\cite{Arnold.D;RS.Falk;Winther.R2000,Hiptmair2002,P.Monk2003}.
\begin{lemma} [Approximation Property of $Q_{\mathcal K}$]\label{lem:app_Q}
There exists a constant $C_A > 0$ such that
$$
\|v_h - Q_{\mathcal K} v_h\|  \leq C_A h^s \|\dd v_h\|\qquad \text{for all } v_h \in \mathcal K_h,
$$ 
where $1/2\leq s\leq 1$ is the regularity index in Lemma \ref{lem:regularity}.
\end{lemma}
\begin{proof}
We shall prove the result by the standard duality argument. 
Let $v = Q_{\mathcal K}v_h$ and $\omega_h \in \mathcal K_h$ be the solution of the equation 
\begin{equation}\label{eq:aux_problem}
\langle \dd \omega_h, \dd \phi_h \rangle = \langle v_h -  v, \phi_h \rangle\qquad \text{for all }\phi_h \in \mathcal K_h.
\end{equation}
Equation \eqref{eq:aux_problem} is well-posed due to the fact that $\langle \dd(\cdot), \dd(\cdot) \rangle$ is an inner-product on the space $\mathcal K_h$, c.f. discrete Poincar\'e inequality \eqref{eq:Poincare_K_h}.
Since both $\mathcal K$ and $\mathcal K_h$ are $L^2$ orthogonal to $\mathcal Z_{0,h}$, the test function space in equation \eqref{eq:aux_problem} can be enlarged, i.e., the solution $\omega_h$ of \eqref{eq:aux_problem} satisfies
$$
\langle \dd \omega_h, \dd \phi_h \rangle = \langle v_h -  v, \phi_h \rangle\qquad \text{for all }\phi_h \in  V_h.
$$ 
We shall use the $L^2$ bounded projection operator $\tilde\Pi_h$ constructed in \cite[page 65-67]{Arnold.D;RS.Falk;Winther.R2006} and \cite{Christiansen;Winther2008}, c.f. \eqref{AFWprojection}. 
Then it holds
\begin{align*}
\|v_h -  v\|^2 & = \langle v_h - v, v_h - \tilde \Pi_h  v + \tilde\Pi_h  v - v \rangle \\
& = \langle  \dd \omega_h, \dd v_h - \dd \tilde\Pi_h  v \rangle + \langle v_h - v, \tilde\Pi_h  v -  v \rangle \\
& =  \langle  \dd \omega_h, \dd v_h - \tilde\Pi_h \dd v_h \rangle + \langle v_h - v, \tilde\Pi_h v -  v \rangle \\
& =  \langle v_h -  v, \tilde\Pi_h  v -  v \rangle \\
& \leq \|v_h -  v\| \|(I - \tilde\Pi_h) v \| \\
& \lesssim h^s \|v_h -  v\| \| v \|_{H^s}.
\end{align*}
By the regularity result, c.f., Lemma \ref{lem:regularity}, Poincar\'e inequality \eqref{eq:Poincare_d}, and the fact that $\delta  v = 0$, it holds
$$
\| v \|_{H^s} \lesssim \| v \| + \|\dd  v \| \lesssim \| \dd v\| = \|\dd v_h\|,
$$
which completes the proof.
\end{proof}

\begin{lemma}[Quasi-Orthogonality of $u$]\label{lem:u_ortho}
Let  $(\sigma,u)\in H_0\Lambda^{-}(\Omega)\times H_0\Lambda(\Omega)$ be the solution of equation \eqref{eq:mixed_formulation_con}  and $(\sigma_h,u_h) \in V_h^{-} \times V_h$ and $(\sigma_H,u_H)\in V_H^{-}\times V_H$ be the solution of equation \eqref{eq:mixed_formulation_dis} on meshes $\mathcal T_h$ and $\mathcal T_H$, respectively. 
Then we have the quasi-orthogonality 
\begin{equation}\label{eq:orth_u}
\|\dd (u - u_h)\|^2 \leq \|\dd (u- u_H)\|^2 - \|\dd (u_h - u_H)\|^2 + 2C_Ah^s\|\dd ^-(\sigma - \sigma_h)\|\|\dd (u_h - u_H)\|,
\end{equation}
where $s$ is the regularity index in Lemma \ref{lem:regularity} and $C_A$ is the constant in Lemma \ref{lem:app_Q}.
\end{lemma}
\begin{proof}
By the second equations of \eqref{eq:mixed_formulation_con} and \eqref{eq:mixed_formulation_dis}, we have
$$
\langle \dd ^-(\sigma - \sigma_h), v_h\rangle + \langle \dd (u - u_h), \dd v_h\rangle = 0\qquad \text{for all } v_h \in V_h.
$$
Then
\begin{align*}
\|\dd (u - u_H)\|^2 & =  \|\dd (u - u_h)\|^2 + \|\dd (u_h - u_H)\|^2 
+ 2 \langle \dd (u - u_h), \dd (u_h - u_H) \rangle \\
& =  \|\dd (u - u_h)\|^2 + \|\dd (u_h - u_H)\|^2 
- 2 \langle \dd ^-(\sigma - \sigma_h), u_h -u_H\rangle.
\end{align*}
The discrete Hodge decomposition implies that
$$
u_h - u_H = \dd ^- \phi_h \oplus^{\bot_{L^2}} \psi_h\qquad \text{ for some } \phi_h \in V_h^{-},\ \psi_h \in \mathcal K_h.
$$
Recall that 
$$
\langle \dd ^-(\sigma - \sigma_h), \dd ^- \zeta_h \rangle = 0\qquad \text{for all }\zeta_h \in V_h^{-}.
$$
Then, by the fact that $\dd ^-(\sigma - \sigma_h) \in \mathcal Z_0$ and $ \mathcal K$ is the $L^2$ orthogonal complement of $\mathcal Z_0$, we have
\begin{align*}
 \langle \dd ^-(\sigma - \sigma_h), u_h -u_H\rangle & =  \langle \dd ^-(\sigma - \sigma_h), \psi_h \rangle 
 =  \langle \dd ^-(\sigma - \sigma_h), \psi_h - Q_{\mathcal K}\psi_h \rangle \\
& \leq  C_A h^s \|\dd ^-(\sigma - \sigma_h)\| \|\dd \psi_h\| \\
&  =   C_A h^s \|\dd ^-(\sigma - \sigma_h)\| \|\dd (u_h - u_H)\|.
\end{align*}
The desired result then follows.
\end{proof}

\section{A Posteriori Error Estimates}

In this section reliability and efficiency of \emph{a posteriori} error estimators will be presented. The Hodge decomposition and the regular decomposition plays an important role in the analysis. 

For any interior face $e \in \mathcal E_h$, let $K^+$ and $K^-$ be two elements in $\mathcal T_h$ sharing $e$. For any quality $\zeta$ satisfying $\zeta|_{K^-} \in H^1\Lambda(K^-)$ and $\zeta|_{K^+} \in H^1\Lambda(K^+)$, let $\zeta^+ = \zeta |_{K^+}$ and $\zeta^- = \zeta |_{K^-}$, then we define $[\zeta]|_e = \zeta^+|_e - \zeta^-|_e$. In the case where $e\subset \partial\Omega$, $[\zeta]$ is simply interpreted as $\zeta|_e$.

\subsection{Reliability} 
In this subsection, the reliability of \emph{a posteriori} error estimators will be proved by using the Hodge decomposition, the regular decomposition, and the commuting quasi-interpolants introduced in Section 2. It should be pointed out that similar estimators have been obtained in~\cite{Demlow2014} but ours are tighter in the sense that $L^2$-norm of $\sigma$ and $u$ are absent.

In this section and following sections, we will always assume $(\sigma,u) \in H_0\Lambda^{-}(\Omega) \times H_0\Lambda(\Omega)$ is the solution of equation \eqref{eq:mixed_formulation_con} and $(\sigma_h,u_h) \in V_h^{-} \times V_h$ the solution of equation \eqref{eq:mixed_formulation_dis}. 
We also assume that $f \in H^1\Lambda^k(\mathcal T_h).$
For any subset $\mathcal M_h \subset \mathcal T_h$, define $\mathcal F_h = \left\{e\subset \partial K:\ \text{for all }K \in \mathcal M_h \right\}$, which denotes faces of elements in $\mathcal M_h$. 
The following error estimators $\eta$ on $\mathcal M_h$ are defined as
\begin{align*}
\eta^2((\sigma_h,u_h),\mathcal M_h) & = \sum\limits_{K \in \mathcal M_h} h_K^2\Big(\|\delta_K(f - \dd ^-\sigma_h)\|_K^2 + \| f - \delta_K \dd u_h - \dd ^-\sigma_h\|_K^2 \Big)\\ 
& + \sum\limits_{e\in\mathcal F_h} h_e\Big(\|[\tr\star(f -  \dd ^- \sigma_h)]\|_e^2 
+ \|[\tr\star \dd u_h]\|_e^2 \Big), \\
\eta^2(\sigma_h,\mathcal M_h) & = \sum\limits_{K\in \mathcal M_h} h_K^2\|\delta_K(f - \dd ^-\sigma_h)\|_K^2 + \sum\limits_{e \in \mathcal F_h}h_e\|[\tr \star (f - \dd ^-\sigma_h)]\|_e^2,
\end{align*}
where $\delta_K$ is the element-wise co-derivative operator defined on $K \in \mathcal T_h$, i.e. $\delta_K$ and $\dd$ satisfying
$$
\langle \dd \omega, \mu \rangle_K = \langle w , \delta_K \mu \rangle_K + \int_{\partial K} \tr \omega \wedge \tr \star \mu, \quad \omega \in \Lambda^{k-1}(K), \ \mu \in \Lambda^k(K).
$$ 
Note that terms in our estimator are only part of that in~\cite{Demlow2014}.

The terms appearing in $\eta$ are reasonable. The Hodge Laplacian equation $\dd ^-\sigma + \delta \dd u = f$ implies that $\delta \dd ^-\sigma = \delta f$. 
The terms $h_K\|\delta_K(f - \dd ^-\sigma_h)\|_K$ and $h_K\|f - \dd ^-\sigma_h - \delta_K \dd u_h \|_K $ measure the elementwise residual in a scaled $L^2$ norm. The terms $h_e^{1/2}\|[\tr \star (f - \dd ^-\sigma_h)]\|_e$ and $h_e^{1/2}\|[\tr\star \dd u_h]\|_e$ measure the inter-element residual with a correct scaling. Note that $(\tr \star f) |_e$ is meaningful as we assume $f \in H^1\Lambda^k(\mathcal T_h)$.

We first give an estimate of $\|\dd ^-(\sigma - \sigma_h)\|$ based on the orthogonality \eqref{eq:or_sigma}.
\begin{lemma}[\emph{A Posteriori} Error Estimates of $\sigma$]\label{lem:sigma}
For any integer $1 \leq k \leq n-1$ and $f \in H^1\Lambda^k(\mathcal T_h)$, there exists a positive constant $C_1 > 0$, such that
$$
\|\dd ^-(\sigma - \sigma_h)\| \leq C_1 \eta(\sigma_h, \mcal T_h).
$$
\end{lemma}
\begin{proof}

For $k = 1$, $H_0\Lambda^0 = H_0^1(\Omega) $, $\dd^0 = \grad$ and $\delta^1 = -\div$. 
It is reduced to \emph{a posteriori} error estimates for the Lagrange element approximation to the standard Poisson equation, cf.~\cite{Verfurth1996,Anisworth;Olden2011}.

For any $\zeta \in H_0\Lambda^{k}$ ($1 \leq k \leq n-2$), by Lemma \ref{lem:Hodge_Regular_dec} and \ref{lem:quasi_inter_1} there exist $\phi \in H_0\Lambda^{k-1}$, $\psi \in H_0\Lambda^{k}$, $\varphi \in H_0\Lambda^{k-1}(\Omega)\cap H^1\Lambda^{k-1}(\mathcal T_h)$, $\omega\in H_0\Lambda^{k}(\Omega)\cap H^1\Lambda^{k}(\mathcal T_h)$ such that 
$$
\zeta = \dd\phi \oplus^{\bot_{L^2}}  \psi,\quad \psi - \Pi_h\psi = \dd\varphi + \omega\quad\text{and} \quad \|\psi\|_{H\Lambda} \lesssim \|\dd ^- \zeta \|,
$$
where recall that $\Pi_h$ is the quasi-interpolation operator in Lemma \ref{lem:quasi_inter}. Therefore, it holds $\dd^-\zeta = \dd^-\psi$. Since for any $K \in \mathcal T_h$, $\omega|_K \in H^1\Lambda^-(K)$, the $\tr(\omega)|_{\partial K}\in L^2(\partial K)$. Using the orthogonality result \eqref{eq:or_sigma}, we have
\begin{align*}
&\langle \dd ^-(\sigma - \sigma_h), \dd ^-\zeta \rangle  =  \langle \dd ^-(\sigma -\sigma_h) , \dd ^-(\psi - \Pi_h\psi)\rangle 
 =  \langle \dd ^-(\sigma -\sigma_h) , \dd ^-\omega\rangle\\
 = & \langle f,\dd^-\omega\rangle -  \langle \dd^-\sigma_h,\dd^-\omega\rangle\\
 = & \sum\limits_{K\in \mathcal T_h} \langle \delta_K(f - \dd^-\sigma_h),\omega\rangle_K + \int_{\partial K}\tr\omega\wedge\tr\star (f - \dd^-\sigma_h) \\
 \leq & \sum\limits_{K \in\mathcal T_h} \|\delta_K(f - \dd ^-\sigma_h)\|_K \|\omega \|_K + \sum\limits_{e \in \mathcal E_h}  \|[\tr\star (f - \dd ^-\sigma_h)]\|_e \|\tr(\omega )\|_e \\
 \lesssim & \left( \sum\limits_{K \in \mathcal T_h} h_K^2\|\delta_K(f - \dd ^-\sigma_h)\|_K^2 + \sum\limits_{e\in\mathcal E_h} h_e\|[\tr\star (f - \dd ^- \sigma_h)]\|_e^2 \right)^{1/2}\|\psi\|_{H\Lambda} \\ 
 \lesssim & \left( \sum\limits_{K \in \mathcal T_h} h_K^2\|\delta_K(f - \dd ^-\sigma_h)\|_K^2 + \sum\limits_{e\in\mathcal E_h} h_e\|[\tr\star (f - \dd ^-\sigma_h)]\|_e^2 \right)^{1/2}\|\dd ^-\zeta\|.
\end{align*}
The desired result then follows by choosing $\zeta = \sigma - \sigma_h$.
\end{proof}

Now we turn to the estimates of the term $\|\dd (u - u_h)\|$. 
\begin{lemma}[\emph{A Posteriori} Error Estimates of $u$]
\label{lem:u}
For an integer $1\leq k \leq n-1$ and $f \in H^1\Lambda^k(\mathcal T_h)$, it holds with a positive constant $C_2$
\begin{equation*}
\|\dd (u - u_h)\|  \leq C_2  \eta((\sigma_h,u_h),\mathcal T_h).
\end{equation*}
\end{lemma} 
\begin{proof}
For any $v \in H_0\Lambda^k$ ($1 \leq k \leq n - 1$), by Lemma \ref{lem:Hodge_Regular_dec} and \ref{lem:quasi_inter_1} there exist $\phi \in H_0\Lambda^{k-1}$, $\psi \in \mathcal K^k$, $\varphi \in H_0\Lambda^{k-1}(\Omega)\cap H^1\Lambda^{k-1}(\mathcal T_h)$, $\omega\in H_0\Lambda^{k}(\Omega)\cap H^1\Lambda^{k}(\mathcal T_h)$ and quasi interpolation operator $\Pi_h$ as in Lemma \ref{lem:quasi_inter} such that 
$$
v = \dd ^-\phi \oplus^{\bot_{L^2}} \psi,\quad \psi - \Pi_h\psi = \dd ^- \varphi + \omega\quad\text{and}\quad \|\dd ^-\varphi\| \lesssim \|\psi\|_{H\Lambda} \lesssim \|\dd v \|.
$$
Then, using the fact  $\dd v = \dd \psi$ and 
$$
\langle \dd^-(\sigma - \sigma_h), v_h\rangle + \langle\dd (u - u_h),\dd v_h\rangle = 0, \quad \dd^-(\sigma - \sigma_h) \bot \, \mathcal K,
$$
it holds
\begin{align*}
\langle \dd (u - u_h), \dd v\rangle & = \langle \dd (u - u_h), \dd \psi\rangle\\  & =  \langle \dd (u - u_h), \dd (\psi - \Pi_h\psi)\rangle + \langle \dd ^-(\sigma - \sigma_h),\psi - \Pi_h \psi\rangle \\
& =  \langle \dd (u - u_h), \dd \omega\rangle + \langle \dd ^-(\sigma - \sigma_h),\omega\rangle + \langle \dd ^-(\sigma - \sigma_h), \dd ^- \varphi \rangle.
\end{align*}
The second equation of \eqref{eq:mixed_formulation_con} implies that
$$
\langle \dd u,\dd \omega \rangle + \langle\dd^-\sigma,\omega\rangle = \langle f,\omega\rangle,
$$
and consequently
\begin{align*}
& \langle \dd (u - u_h), \dd v\rangle = \langle f - \dd^-\sigma_h,\omega\rangle - \langle \dd u_h,\dd\omega\rangle + \langle \dd^-(\sigma - \sigma_h),\dd^-\varphi\rangle \\
= &  \sum\limits_{K \in \mathcal T_h} \Big( \langle f - \dd ^-\sigma_h -\delta_K \dd u_h, \omega  \rangle_K - 
 \int_{\partial K}\tr (\omega ) \wedge \tr\star \dd u_h  \Big) \\
 &\qquad  + \langle \dd ^-(\sigma - \sigma_h), \dd ^- \varphi \rangle \\
 \leq &  \sum\limits_{K\in\mathcal T_h} \|f - \delta_K \dd u_h - \dd ^-\sigma_h\|_K \|\omega \|_K + \sum\limits_{e \in \mathcal E_h} \|[\tr\star \dd u_h]\|_e \|\tr(\omega )\|_e \\
&  \qquad + \|\dd ^-(\sigma - \sigma_h)\|\|\psi\|_{H\Lambda} \\
 \lesssim &  \left( \sum\limits_{K \in \mathcal T_h} h_K^2\| f - \delta_K \dd u_h - \dd ^-\sigma_h\|_K^2 + \sum\limits_{e \in \mathcal E_h} h_e \|[\tr\star \dd u_h]\|_e^2 \right)^{1/2} \|\psi\|_{H\Lambda} \\
&   \quad + \|\dd ^-(\sigma - \sigma_h)\|\|\psi\|_{H\Lambda}.
\end{align*}
Here, because $\omega|_K \in H^1\Lambda(K)$, the trace term $\|\tr(\omega )\|_e$ is meaningful and bounded by $\|\psi\|_{H\Lambda}$.
By the fact that 
$
\|\psi\|_{H\Lambda} \lesssim  \|\dd v\|,
$
and Lemma \ref{lem:sigma}, the desired result follows.
\end{proof}

As a summary of Lemmas \ref{lem:sigma} and \ref{lem:u}, we have the reliability result. 
\begin{theorem}[Reliability]
\label{them:reliability}
Let $(\sigma,u)\in H_0\Lambda^{k-1}(\Omega)\times H_0\Lambda^k(\Omega)$ and $(\sigma_h , u_h) \in V_h^{k-1}\times V_h^k$ ($1 \leq k \leq n-1$) be the solutions of equations \eqref{eq:mixed_formulation_con} and \eqref{eq:mixed_formulation_dis}, respectively. Assume that $f \in H^1\Lambda^k(\mathcal T_h)$. Then, it holds
$$
\|(\sigma - \sigma_h,u-u_h)\| \leq C_3 \eta((\sigma_h,u_h),\mathcal T_h).
$$
\end{theorem}


\subsection{Efficiency} 
The efficiency of the estimator $\eta$ is obtained by using the standard bubble function technique~\cite{Demlow2014,Verfurth1996}. We will list the efficiency results here and refer to~\cite{Demlow2014} for details. For any subset $\mathcal M_h \subset \mathcal T_h$, define the data oscillation as
\begin{equation}
\label{eq:oscillation}
\osc^2(f,\mathcal M_h)  =  \sum\limits_{K \in \mathcal M_h} h_K^2 \left( \|f - Q^k_h f\|_K^2  + \|\delta_Kf - Q^{k-1}_h\delta_K f\|_K^2 \right),
\end{equation}
where $Q^k_h: L^2\Lambda^k(\Omega) \mapsto V_h^k$ is the $L^2$ projection. We have the following efficiency results
\begin{theorem}[Efficiency~\cite{Demlow2014}]
Let $(\sigma,u)\in H_0\Lambda^{-}(\Omega)\times H_0\Lambda(\Omega)$ and $(\sigma_h,u_h) \in V_h^{-} \times V_h$ be the solutions of equations \eqref{eq:mixed_formulation_con} and \eqref{eq:mixed_formulation_dis}, respectively. Then there exists positive constants $C_4$ and $C_{\osc}$ such that
$$
C_4 \eta((\sigma_h,u_h),\mathcal T_h) \leq  \|(\sigma - \sigma_h,u - u_h)\| + C_{\osc}\osc(f,\mathcal T_h).
$$
\end{theorem}

\section{Convergence of AMFEM}

In this section, we shall present our algorithm and prove its convergence  using approaches developed in~\cite{Dorfler1996,Morin;Nochetto;Siebert;2000,Morin;Nochetto;Siebert;2002,Stevenson2007,Chen;Holst;Xu2009,Feischl;Fuhrer;Pratetorius2014}. Such convergence results are well established for symmetric positive definite problems and nested finite element spaces. For saddle point systems and/or non-nested spaces, results are less standard~\cite{Chen;Holst;Xu2009,Huang;Xu2012,Chen;Xu;Zou2012}. Comparing with existing results, e.g., for Maxwell's equations~\cite{Chen;Xu;Zou2009,Zhong;Chen;Shu;Xu2012,Chen;Xu;Zou2012,Duan;Qiu;Tan;Zheng2016}, we obtain a convergence proof without assuming that the initial mesh size is small enough. 

In what follows, we replace the dependence on the actual mesh $\mathcal T$ by the iteration counter $l$, and let $h_l = \max\limits_{K\in \mathcal T_l} \{ h_K\}$ be the mesh size of $\mcal T_l$.

\noindent\smallskip
\breakline
\textbf{Algorithm 1.} Given an initial mesh $\mathcal T_0$, a marking parameter $0 < \theta < 1$, and a stopping tolerance ${\rm tol} > 0$. 
Solve equation \eqref{eq:mixed_formulation_dis} on mesh $\mathcal T_0$ to get the solution $(\sigma_0,u_0)$ and compute the error estimator $\eta_0 = \eta((\sigma_0,u_0),\mathcal T_0)$. Set $l = 0$ and iterate.

While $\eta_l > {\rm tol}$, do

\begin{description}
  \item[Step 1]  Solve equation \eqref{eq:mixed_formulation_dis} on mesh $\mathcal T_l$ to get the solution $(\sigma_l,u_l)$.
  \item[Step 2] Compute the error estimators $\eta((\sigma_l,u_l),\mathcal T_l)$ and $\eta(\sigma_l,\mathcal T_l)$.
  \item[Step 3] 
Mark an element set $\mathcal M_l$, such that
\begin{align}
  \label{eq:mark_sigma}
  \eta^2(\sigma_{l},\mathcal M_{l})  &\geq \theta \eta^2(\sigma_{l},\mathcal T_{l}),\\
  \label{eq:mark}
  \eta^2((\sigma_l,u_l),\mathcal M_l) & \geq \theta \eta^2((\sigma_l,u_l),\mathcal T_l) .
\end{align}
  
    \item[Step 4]  Refine each element $K\in \mathcal M_{l}$  by the newest          vertex bisection, and make some necessary completion to get a conforming and shape regular mesh $\mathcal T_{l+1}$.
\end{description}
\breakline

\medskip

In the marking step, we mark an element set $\mathcal M_l$ so that both error estimators $\eta(\sigma_l,\mathcal M_l)$ and $\eta((\sigma_l,u_l),\mathcal M_l)$ satisfy the D\"orfler marking. A possible choice of $\mathcal M_l$ can be obtained by marking separately for each and then take the union. The requirement \eqref{eq:mark_sigma} is to ensure the contraction of $\eta(\sigma_l)$ so that after few steps, we can obtain a quasi-orthogonality for the total error. We could do so by assuming the initial mesh size is small enough. By including \eqref{eq:mark_sigma}, we successfully remove such condition.

The following contraction result for $\eta(\sigma)$ can be proved using the arguments in~\cite[Theorem 4.1]{Feischl;Fuhrer;Pratetorius2014}. For completeness, we give a proof simplified by the orthogonality.
\begin{theorem}[Convergence of $\eta(\sigma)$]\label{them:con_sigma}
Let $(\sigma,u)\in V^{-}\times V$ be the solution of equation \eqref{eq:mixed_formulation_con}  and for any positive integers $l\text{ and }m$, $(\sigma_{l+m},u_{l+m}) \in V_{l+m}^{-} \times V_{l+m}$ and $(\sigma_{l},u_{l})\in V_{l}^{-}\times V_{l}$ be the solutions of equation \eqref{eq:mixed_formulation_dis} on meshes $\mathcal T_{l+m}$ and $\mathcal T_{l}$, respectively. Then there exist constants $ 0 < \varrho < 1$ and $C_5 > 0$ such that
\begin{equation}\label{etacontaction}
 \eta^2(\sigma_{l+m},\mathcal T_{l+m}) \leq C_5^2 \varrho^m \eta^2(\sigma_{l},\mathcal T_{l}).
\end{equation}
\begin{proof}
When D\"orfler marking is satisfied, the following result can  be proved, c.f.~\cite[Corollary 3.4]{Cascon;Kreuzer2008} and \cite[Lemma 3.1]{Feischl;Fuhrer;Pratetorius2014}: There exist constants $0 < q_c < 1$ and $C_{\theta}$ dependending only on $\theta$, such that
\begin{equation}\label{eq:continuous_sigma}
\eta^2(\sigma_{i+1},\mathcal T_{i+1}) \leq q_c \eta^2(\sigma_i,\mathcal T_i) +  C_{\theta}\|d(\sigma_{i+1} - \sigma_i)\|^2.
\end{equation} 
Therefore, by Lemmas \ref{lem:sigma_ortho} and \ref{lem:sigma}, for any $N \geq l+1$, it holds
\begin{align*}
\sum\limits_{i = l+1}^N\eta^2(\sigma_i,\mathcal T_i) & \leq \sum\limits_{i = l+1}^N \left [ q_c\eta^2(\sigma_{i - 1},\mathcal T_{i - 1}) + C\|\dd ^-(\sigma_i - \sigma_{i - 1})\|^2 \right ] \\
& \leq q_c \sum\limits_{i = l}^{N -1}  \eta^2(\sigma_{i },\mathcal T_{i })  + C \|\dd ^-(\sigma - \sigma_{l})\|^2\\
& \leq q_c \sum\limits_{i = l}^{N-1}  \eta^2(\sigma_{i },\mathcal T_{i })  + C C_1^2 \eta^2(\sigma_l,\mathcal T_l).
\end{align*}
Here, in the second inequality, we have used the orthogonality to get 
\begin{align*}
\sum\limits_{i = l+1}^N \|\dd^-(\sigma_i - \sigma_{i-1})\|^2 & = \|\dd^-(\sigma_N - \sigma_l)\|^2  = \|\dd^-(\sigma - \sigma_l)\|^2 - \|\dd^-(\sigma - \sigma_N)\|^2 \\
& \leq \|\dd^-(\sigma - \sigma_l)\|^2.
\end{align*}
Then, rearranging the terms and with the arbitrary choice of $N$, we obtain
\begin{align*}
\sum\limits_{i = l+1}^{\infty}\eta^2(\sigma_i,\mathcal T_i)  \leq \tilde C\eta^2(\sigma_l,\mathcal T_l) \qquad \text{for all positive integer } l,
\end{align*}
where $\tilde C = (q_c + CC_1^2)/(1 - q_c)$.
Therefore, we get
$$
(1 + \tilde C^{-1}) \sum\limits_{i = l + 1}^{\infty}\eta^2(\sigma_i,\mathcal T_i) \leq \sum\limits_{i = l + 1}^{\infty}\eta^2(\sigma_i,\mathcal T_i) + \eta^2(\sigma_l,\mathcal T_l) = \sum\limits_{i = l }^{\infty}\eta^2(\sigma_i,\mathcal T_i).
$$ 
By the inductive method, we have
\begin{align*}
\eta^2(\sigma_{l+m},\mathcal T_{l+m}) & \leq \sum\limits_{i = l + m}^{\infty}\eta^2(\sigma_i,\mathcal T_i) \leq (1 + \tilde C^{-1})^{-m} \sum\limits_{i = l }^{\infty}\eta^2(\sigma_i,\mathcal T_i) \\
& \leq (1 + \tilde C) (1 + \tilde C^{-1})^{-m}\eta^2(\sigma_l,\mathcal T_l).
\end{align*}
Let $C_5^2 = 1 + \tilde C$ and $\varrho = (1 + \tilde C^{-1})^{-1}$, then the desired result follows.
\end{proof}
\end{theorem}

We aim to prove that Algorithm 1 will be terminated in finite steps. We will first use the contraction of $\eta(\sigma_l,\mathcal T_l)$, see Theorem \ref{them:con_sigma}, to prove a quasi-orthogonality result and then derive the convergence of a subsequence. For easy of notation, denote by
\begin{align*}
e_l^2  &= \|(\sigma - \sigma_l,u - u_l)\|^2,\\ 
E_{l , m}^2& = \|(\sigma_{l+m} - \sigma_l,u_{l + m} - u_l)\|^2,\\ 
\eta_l^2  &= \eta^2((\sigma_l,u_l),\mathcal T_l).
\end{align*}

\begin{lemma}\label{lem:results}
For any positive integer $m$, there exist constants $0 < \mu < 1$ , $C_6 > 0$ and $C_{\theta} > 0$, such that
\begin{align}
\label{eq:re}  
& e_l^2  \leq  C_6\eta_l^2 ,\\
\label{eq:quasi_ortho} & e_{l+m}^2 \leq e_l^2 - (1 - \epsilon_{l,m})E_{l , m}^2  + \epsilon_{l,m} \eta_l^2,\\
\label{eq:eta} &  \eta_{l+m}^2 \leq   \mu \eta^2_l + C_{\theta} E_{l , m}^2,
\end{align}
where $\epsilon_{l,m} = C_AC_1C_5h_{l+m}^s\varrho^{m/2}$.
\end{lemma}
\begin{proof}
Equation \eqref{eq:re} is the reliability of $\eta_l$ with $C_6 = C_3^2$, see Theorem \ref{them:reliability}.

By Lemmas \ref{lem:sigma_ortho} and \ref{lem:u_ortho}, and equation \eqref{etacontaction}, we have
\begin{align*}
e_{l+m}^2 & \leq e_l^2 - E_{l , m}^2  + 2C_Ah_{l+m}^s\|\dd ^-(\sigma - \sigma_{l+m} )\| \|\dd (u_{l+m} - u_l)\| \\
& \leq e_l^2 - E_{l , m}^2  + 2C_A C_1 h_{l+m}^s \eta(\sigma_{l+m},\mathcal T_{l+m}) \|\dd (u_{l+m} - u_l)\| \\
& \leq e_l^2 - E_{l , m}^2  + 2\epsilon_{l,m} \eta(\sigma_l,\mathcal T_l) \|\dd (u_{l+m} - u_l)\| \\
&\leq e_l^2 - (1 - \epsilon_{l,m})E_{l , m}^2  + \epsilon_{l,m}\eta_l^2.
\end{align*}
Thus \eqref{eq:quasi_ortho}  is proved.

Inequality \eqref{eq:eta} is a variation of \eqref{eq:continuous_sigma},
and hence can be obtained by the same argument.
\end{proof}

Due to the fact that $\varrho \in (0,1)$ and $h_l$ is non-increasing, $\epsilon_{l,m}$ is non-increasing with respect to the index $l$ and geometrically decreasing with respect to the index $m$.

\begin{theorem}[Convergence]
Let $(\sigma,u)\in H_0\Lambda^{-}(\Omega)\times H_0\Lambda(\Omega)$ be the solution of equation \eqref{eq:mixed_formulation_con}. Let $k,m$ be two non-negative integers and $(\sigma_{(k+1)m},u_{(k+1)m}) \in V_{(k+1)m}^{-} \times V_{(k+1)m}$ and $(\sigma_{km},u_{km})\in V_{km}^{-}\times V_{km}$ be the solutions of equation \eqref{eq:mixed_formulation_dis} on meshes $\mathcal T_{(k+1)m}$ and $\mathcal T_{km}$, respectively. Then for $m$ large enough s.t. 
\begin{equation}\label{eq:m}
\epsilon_{0,m} < \frac{1 - \mu}{1 - \mu + C_{\theta}},
\end{equation}
there exist constants $\alpha_m > 0$, $0 < \rho_m < 1$ satisfying  
$$
e_{(k+1)m}^2 + \alpha_m \eta_{(k+1)m}^2 \leq \rho_m (e_{km}^2 + \alpha_m \eta_{km}^2 ),
$$
and consequently
$$
e_{km}^2 + \alpha_m \eta_{km}^2 \leq \rho_m^k (e_{0}^2 + \alpha_m \eta_{0}^2 ).
$$
\end{theorem}
\begin{proof}
The existence of $m$ satisfying \eqref{eq:m} is obtained by the fact that $\lim_{m\to \infty}\epsilon_{0,m} = 0$.
For any $\alpha > 0$, equations \eqref{eq:quasi_ortho} and \eqref{eq:eta} imply that
\begin{align*}
e_{(k+1)m}^2 + \alpha \eta^2_{(k+1)m} & \leq e_{km}^2 - (1-\epsilon_{km,m})E_{(k+1)m, km}^2   \\
&\quad + \left(\epsilon_{km,m} + \alpha\mu\right) \eta_{km}^2+ \alpha C_{\theta}E_{(k+1)m, km}^2 .
\end{align*}
Choosing $\alpha = \alpha_m : = C_{\theta}^{-1}(1 - \epsilon_{0,m})>0$, as $\epsilon_{0,m}<1$ by \eqref{eq:m}. Using the inequality $\epsilon_{km,m} \leq \epsilon_{0,m}$, we can get
\begin{align*}
e_{(k+1)m}^2 + \alpha_m \eta^2_{(k+1)m} & \leq e_{m}^2 + \left(\epsilon_{0,m} + \alpha_m \mu\right) \eta_{km}^2.
\end{align*}
Let $\rho$ be a number in $(0,1)$ to be determinated in a moment. We then have
\begin{align*}
e_{(k+1)m}^2 + \alpha_m \eta^2_{(k+1)m} & \leq \rho e_{km}^2 +  (1 - \rho)e_{km}^2 + \left(\epsilon_{0,m} + \alpha_m \mu\right) \eta_{km}^2\\
& \leq \rho e_{km}^2 + \Big[ C_6 \left(1  - \rho\right) + \epsilon_{0,m} + \alpha_m \mu \Big]\eta_{km}^2 \\
& = \rho \left\{ e_{km}^2 + \frac{C_6 \left(1  - \rho\right) + \epsilon_{0,m} + \alpha_m \mu}{\rho} \eta_{km}^2 \right\}.
\end{align*}
This suggests us to choose $\rho$ such that
\begin{equation}\label{alpha_eq}
\alpha_m = \frac{C_6 \left(1  - \rho\right) + \epsilon_{0,m} + \alpha\mu}{\rho}. 
\end{equation}
Solving \eqref{alpha_eq} to get
$$
\rho = \rho_m = \frac{C_6 + \alpha\mu + \epsilon_{0,m} }{C_6 + \alpha}.
$$
It is obvious that $\rho_m >0$. Now we will show $\rho_m < 1$. Recall $\alpha_m = C_{\theta}^{-1}(1 - \epsilon_{0,m})$, thus $\rho_m < 1$ is equivalent to 
\begin{equation}\label{eq:rholeq1}
 (1 - \mu + C_{\theta} ) \epsilon_{0,m} < 1 - \mu,
\end{equation}
which can be deduced from \eqref{eq:m}. Then the desired result follows.
\end{proof}

We conclude this section with the following remarks.
\begin{remark}\label{re:h0}\rm
\begin{enumerate}
  \item We do not need to know the value of $m$ in practice. We just  iterate, a subsequence of $\{\eta(\sigma_l,u_l)\}|_{l = 1}^{\infty}$, with bounded index gap, will converge to zero which implies that Algorithm 1 will terminate in finite steps.  
 
 \item Equation \eqref{eq:m} implies that $m$ may not be a big number. The proof can be easily modified to the sub-sequence with index $l+ k m$ for fixed positive integers $l,m$ and varying $k=0,1,2,\ldots$. Then the constraint \eqref{eq:m} will be posed for $\epsilon_{l,m}$ which might hold for $m = 0$ if $h_l$ is already small enough. Indeed in the existing work for Maxwell's equations~\cite{Chen;Xu;Zou2009,Zhong;Chen;Shu;Xu2012,Chen;Xu;Zou2012,Duan;Qiu;Tan;Zheng2016}, the quasi-orthogonality is obtained by assuming the initial grid size is small enough, so that \eqref{eq:m} holds with $m = 0$.  We include the marking \eqref{eq:mark_sigma} for $\sigma$ to guarantee \eqref{eq:m} holds without this assumption. 
\end{enumerate}
\end{remark}
\smallskip

\begin{remark}[Optimality with the fine enough initial mesh assumption]\rm From Remark \ref{re:h0}, we can see that marking step \eqref{eq:mark_sigma} for $\sigma$ is only used to remove the assumption that the initial mesh $\mathcal T_0$ is fine enough. If we accept this assumption, then we can remove the marking step \eqref{eq:mark_sigma} for $\sigma$ and only mark a minimal elements set using marking strategy \eqref{eq:mark}. By assuming $h_0$ is small enough, and by Lemma \ref{lem:sigma_ortho} and \ref{lem:u_ortho}, we will have the quasi-orthogonality 
\begin{equation}\label{eq:quasi_ortho_h_small}
(1 - C_Ah_{l+1}^s) e_{l+1}^2 \leq e_l^2 - (1 - C_A h_{l+1}^s)E_l^2.
\end{equation}
Using \eqref{eq:quasi_ortho_h_small} and following the same line as the proof of Theorem 4.1 in~\cite{Feischl;Fuhrer;Pratetorius2014}, we can prove the contraction of $\eta_l$.

In~\cite{Falk;Winther2014}, Falk and Winther constructed a localized projector and proved some stable and approximation properties of this projector for non-homogenous boundary conditions. They also briefly remarked that such properties can be extended to the case with homogenous boundary conditions. A localized upper bound can be established using Falk and Winther's interpolant assuming that the expected properties of the Falk and Winther's interpolant hold also for homogeneous boundary conditions. The optimality then follows by using the standard techniques, such as in~\cite{Cascon;Kreuzer2008,Feischl;Fuhrer;Pratetorius2014,Carstensen;Feischl;Praetoriu2014} with the assumption: the initial mesh size is small enough. $\Box$
\end{remark}

\section{Examples}

In this section we translate our results into standard notation of vector calculus for $n = 3$. The case $k = 1$ or $k = 2$ corresponds to the vector Laplacian operator $\curl\curl - \grad\div$ with different boundary conditions and $H_0\Lambda^0(\Omega) = H_0^1(\Omega)$, $H_0\Lambda^1(\Omega) = \bs H_0(\curl;\Omega)$, $H_0\Lambda^2(\Omega) = \bs H_0(\div;\Omega)$, and $H_0\Lambda^3(\Omega) = L_0^2(\Omega)$. The discrete space $V_h^0$ is the well-known Lagrange element space, i.e., continuous and piecewise polynomials space, $V_h^1$ is the edge element space,
$V_h^2$ is the face element space,
and $V_h^3$ is discontinuous and piecewise polynomial space.

\subsection{Vector Laplacian: $k = 1$} Equation \eqref{eq:mixed_formulation_con} can be written as: Find $(\sigma,u) \in H_0^1(\Omega) \times \bs H_0(\curl;\Omega)$ such that 
\begin{equation}\label{eq:k1}
\left\{ 
\begin{array}{llllll}
\langle \sigma ,\tau \rangle & - & \langle \grad \tau, u\rangle & = & 0 &\text{for all }\tau \in H_0^1(\Omega)\\
\langle \grad\sigma, v\rangle & + & \langle \curl u, \curl v \rangle & = & \langle f, v \rangle & \text{for all } v \in \bs H_0(\curl;\Omega)
\end{array}  
\right.
\end{equation}
and the discrete formulation is: Find $(\sigma_h, u_h) \in V_h^0 \times V_h^1$ such that  
\begin{equation}\label{eq:k1_dis}
\left\{ 
\begin{array}{llllll}
\langle \sigma_h ,\tau_h \rangle & - & \langle \grad \tau_h,u_h\rangle & = & 0 &\text{for all }\tau_h \in V_h^0\\
\langle \grad\sigma_h, v_h\rangle & + & \langle \curl u_h, \curl v_h \rangle & = & \langle f, v_h \rangle & \text{for all } v_h \in V_h^1
\end{array}  
\right.
\end{equation}
The \emph{a posteriori} error estimator $\eta$ is defined as
\begin{align*}
\eta^2((\sigma_h,u_h),\mathcal T_h) & =  \sum\limits_{K \in \mathcal T_h} h_K^2\left( \|\div_K(f - \grad\sigma_h)\|_K^2 + \|f - \curl_K\curl u_h - \grad\sigma_h\|_K^2 \right) \\
&\quad + \sum\limits_{e \in \mathcal E_h} h_e\left(  \|[f - \grad \sigma_h\cdot \bs n]\|_e^2 + \| [\curl u_h \times \bs n] \|_e^2 \right).
\end{align*}
The data oscillation is defined as
$$
\osc^2(f,\mathcal T_h) = \sum\limits_{K \in \mathcal T_h} h_K^2\left( \|f - Q_hf\|_K^2 + \| \div f - Q_h \div f\|_K^2 \right).
$$

The efficiency and reliability of the \emph{a posteriori} error estimators is given as follows:
\begin{align*}
&  \|\grad(\sigma - \sigma_h)\|^2 + \|\curl(u - u_h)\|^2  \leq C_3^2 \eta^2((\sigma_h,u_h),\mathcal T_h),\\
&C_4^2 \eta^2((\sigma_h,u_h),\mathcal T_h)  \leq  \|\grad(\sigma - \sigma_h)\|^2 + \|\curl(u - u_h)\|^2 + C\osc^2(f,\mathcal T_h),
\end{align*}
and the convergence result is  
\begin{eqnarray*}
& &\|\grad(\sigma - \sigma_{(k+1)m})\|^2 + \|\curl(u - u_{(k+1)m})\|^2 + \alpha_m\eta^2((\sigma_{(k+1)m},u_{(k+1)m}),\mathcal T_{(k+1)m})\\
& \leq & \rho_m \Big( \|\grad(\sigma - \sigma_{km})\|^2 + \|\curl(u - u_{km})\|^2 + \alpha_m\eta^2((\sigma_{km},u_{km}),\mathcal T_{km}) \Big).
\end{eqnarray*}

\subsection{Maxwell's equations}

We consider a prototype of Maxwell's equation with divergence free constraint
\begin{equation}\label{eq:maxwell}
\curl\curl u = f,\ \div u = 0,\ \text{ in } \Omega,\qquad  u\times n = 0\text{ on } \partial\Omega.
\end{equation}
For the well-posedness of equation \eqref{eq:maxwell},  the right hand side $f$ should   satisfy $\div f = 0$.

The mixed form of equation \eqref{eq:maxwell} is: Find $(\sigma,u) \in H_0^1(\Omega) \times \bs H_0(\curl;\Omega)$ such that
\begin{equation}\label{eq:maxwell_mix}
\left\{ 
\begin{array}{llllll}
&  & \langle u, \grad \tau\rangle & = & 0 &\text{for all }\tau \in H_0^1(\Omega),\\
\langle \curl u, \curl v \rangle & + & \langle \grad\sigma, v\rangle  & = & \langle f, v \rangle & \text{for all } v \in \bs H_0(\curl;\Omega).
\end{array}  
\right.
\end{equation}
The discrete form of equation \eqref{eq:maxwell_mix} is: Find $(\sigma_h,u_h) \in V_h^0 \times V_h^1$ such that
\begin{equation}\label{eq:maxwell_mix_dis}
\left\{ 
\begin{array}{llllll}
&  & \langle u_h, \grad \tau_h\rangle & = & 0 &\text{for all }\tau_h \in V^0_h \\
\langle \curl u_h, \curl v_h \rangle & + & \langle \grad\sigma_h, v_h\rangle  & = & \langle f, v_h \rangle & \text{for all } v_h \in V_h^1.
\end{array}  
\right.
\end{equation}

The system \eqref{eq:maxwell_mix} is different with \eqref{eq:k1} in that the term $\langle \sigma,\tau \rangle$ is missing in the first equation of \eqref{eq:maxwell_mix}. Because of that, the equation $\langle u,\grad \tau\rangle = 0$ implies that $u \in \mathcal K$, and thus $\|\curl u\|$ alone is a norm. The orthogonality \eqref{eq:orth_sigma} for $\sigma$ and quasi-orthogonality \eqref{eq:orth_u} for $u$ are still true. Following the same line as that of the Hodge Laplacian of $k = 1$, we obtain the following convergence result
\begin{align*}
&\|\grad(\sigma - \sigma_{(k+1)m})\|^2 + \|\curl(u - u_{(k+1)m})\|^2 + \alpha_m\eta^2((\sigma_{(k+1)m},u_{(k+1)m}),\mathcal T_{(k+1)m})\\
 \leq & \rho_m \Big( \|\grad(\sigma - \sigma_{km})\|^2 + \|\curl(u - u_{km})\|^2 + \alpha_m \eta^2((\sigma_{km},u_{km}),\mathcal T_{km}) \Big).
\end{align*}
We include the $\sigma$ part to ensure the contraction of the total error without the assumption on the initial mesh size.

It should be pointed out that, in~\cite{Chen;Xu;Zou2012,Duan;Qiu;Tan;Zheng2016}, the convergence of AMFEM was given for Maxwell's equations with similar error estimator and quasi-orthogonality results under the assumption that the initial mesh is fine enough. We remove this assumption in our approach.

\subsection{Vector Laplacian: $k = 2$} Equation \eqref{eq:mixed_formulation_con} becomes: Find $(\sigma,u) \in \bs H_0(\curl;\Omega)    \times \bs H_0(\div;\Omega)$ such that 
\begin{equation}
\left\{ 
\begin{array}{llllll}
\langle \sigma ,\tau \rangle & - & \langle \curl\tau,u\rangle & = & 0 &\text{for all }\tau \in \bs H_0(\curl;\Omega)\\
\langle \curl\sigma, v\rangle & + & \langle \div u, \div v \rangle & = & \langle f, v \rangle & \text{for all } v \in \bs H_0(\div;\Omega)
\end{array}  
\right.
\end{equation}
and the discrete formulation is: Find $(\sigma_h, u_h) \in V_h^1 \times V_h^2$ such that  
\begin{equation}
\left\{ 
\begin{array}{llllll}
\langle \sigma_h ,\tau_h \rangle & - & \langle \curl\tau_h,u_h\rangle & = & 0 &\text{for all }\tau_h \in V_h^1\\
\langle \curl\sigma_h, v_h\rangle & + & \langle \div u_h, \div v_h \rangle & = & \langle f, v_h \rangle & \text{for all } v_h \in V_h^2
\end{array}  
\right.
\end{equation}
The \emph{a posteriori} error estimator $\eta$ is defined as
\begin{align*}
\eta^2((\sigma_h,u_h),\mathcal T_h) & =  \sum\limits_{K \in \mathcal T_h} h_K^2\left( \|\curl_K(f - \curl\sigma_h)\|_K^2 + \|f + \grad_K\div u_h - \curl\sigma_h\|_K^2 \right) \\
&\quad + \sum\limits_{e \in \mathcal E_h} h_e\left(  \|[f - \curl \sigma_h\times \bs n]\|_e^2 + \| [\div u_h  \bs n] \|_e^2 \right).
\end{align*}
The data oscillation is defined as
$$
\osc^2(f,\mathcal T_h) = \sum\limits_{K \in \mathcal T_h} h_K^2\left( \|f - Q_hf\|_K^2 + \| \curl f - Q_h^-\curl f\|_K^2 \right).
$$

The efficiency and reliability of the \emph{a posteriori} error estimators is given as follows:
\begin{align*}
& \|\curl(\sigma - \sigma_h)\|^2 + \|\div(u - u_h)\|^2  \leq C_3^2\eta^2((\sigma_h,u_h),\mathcal T_h),\\
& C_4^2 \eta^2((\sigma_h,u_h),\mathcal T_h)  \leq  \|\curl(\sigma - \sigma_h)\|^2 + \|\div(u - u_h)\|^2 + C\osc^2(f,\mathcal T_h),
\end{align*}
and the convergence result is
\begin{eqnarray*}
& &\|\curl(\sigma - \sigma_{(k+1)m})\|^2 + \|\div(u - u_{(k+1)m})\|^2 + \alpha_m\eta^2((\sigma_{(k+1)m},u_{(k+1)m}),\mathcal T_{(k+1)m})\\
& \leq & \rho_m \Big( \|\curl(\sigma - \sigma_{km})\|^2 + \|\div(u - u_{km})\|^2 + \alpha_m \eta^2((\sigma_{km},u_{km}),\mathcal T_{km}) \Big).
\end{eqnarray*}

It seems that there are no convergence results of the adaptive $\bs H(\curl) \times \bs H(\div)$ mixed finite element method in the literature.


\section*{Acknowledgments}
We would like to thank the unknown referees for carefully reading our manuscript and for giving constructive comments which substantially helped improving the quality of the paper. We would also like to thank Professor Alan Demlow for valuable discussion and suggestions.

\bibliographystyle{siamplain}

\end{document}